\documentclass[final,1p]{elsarticle}
\usepackage{amsmath}

\usepackage{amssymb}
\usepackage{url}

\usepackage[german,english]{babel}
\newtheorem{theorem}{Theorem}
\newtheorem{corollary}[theorem]{Corollary}
\newtheorem{lemma}[theorem]{Lemma}
\newtheorem{fact}[theorem]{Fact}
\newtheorem{proposition}[theorem]{Proposition}

\newtheorem{remark}{Remark}
\newtheorem{claim}{Claim}
\newproof{proof}{Proof}
\newproof{poc}{Proof of Claim}

\newcommand{\N}{\mathbb N}
\newcommand{\Z}{\mathbb Z}
\newcommand{\G}{\mathbb G}
\newcommand{\V}{\mathbb V}
\newcommand{\A}{\mathbb A}
\newcommand{\F}{\mathbb F}
\newcommand{\K}{\mathbb K}
\newcommand{\R}{\mathbb R}
\newcommand{\Pp}{\mathbb P}

\newcommand{\defineAs}{=}
\newcommand{\threshold}{4r-8}

\newcommand{\fp}{\F_{\hskip-0.7mm p}}
\newcommand{\fq}{\F_{\hskip-0.7mm q}}
\newcommand{\fqk}{\F_{\hskip-0.7mm q^k}}
\newcommand{\fqd}{\F_{\hskip-0.7mm q^d}}
\newcommand{\fql}{\F_{\hskip-0.7mm q^\ell}}
\newcommand{\fqll}{\F_{\hskip-0.7mm q^l}}
\newcommand{\cfq}{\overline{\F}_{\hskip-0.7mm q}}

\usepackage[active]{srcltx} 

\def\ifm#1#2{\relax \ifmmode#1\else#2\fi}

\newcommand{\klk}    {\ifm {,\ldots,} {$,\ldots,$}}

\newcommand{\plp}    {\ifm {+\cdots+} {$+\ldots+$}}

\journal{Journal of Number Theory}

\begin{document}
\begin{frontmatter}
\title{The
  number of reducible space curves over a finite field}%
  \author[UNGS,conicet]{Eda Cesaratto\corref{cor1}\fnref{fn1}}
\ead{ecesarat@ungs.edu.ar}

\cortext[cor1]{Corresponding author}

 \author[Bit.uni]{Joachim von zur Gathen\fnref{fn1}}
\ead{gathen@bit.uni-bonn.de}

\author[UNGS,conicet]{Guillermo Matera\fnref{fn1}}%
\ead{gmatera@ungs.edu.ar}
\address[UNGS]{Instituto del Desarrollo Humano, Universidad Nacional
  de Gene\-ral Sarmiento,\\ J.M. Guti\'errez 1150 (B1613GSX) Los
  Polvorines, Buenos Aires, Argentina}
  \address[conicet]{National
  Council of Science and Technology (CONICET), Ar\-gentina}

\address[Bit.uni]{B-IT, Universit\"at Bonn, D-53113 Bonn, Germany}

\fntext[fn1]{Joachim von zur Gathen was supported by the B-IT
Foundation and the Land Nordrhein-Westfalen.\\ Eda Cesaratto and
Guillermo Matera were partially supported by
  grant PIP 11220090100421
  CONICET.}%
%

\begin{abstract}
  ``Most'' hypersurfaces in projective space are irreducible, and
  rather precise estimates are known for the probability that a
  random hypersurface over a finite field is reducible.  This paper
  considers the parametrization of space curves by the appropriate
  Chow variety, and provides bounds on the probability that a random
  curve over a finite field is reducible.
\end{abstract}
\end{frontmatter}
%
%
\section{Introduction}\label{sec:intro}

The Prime Number Theorem and a well-known result of Gau\ss ~describe
the density of primes and of irreducible univariate polynomials over a
finite field, respectively. ``Most'' numbers are composite, and
``most'' polynomials reducible. The latter changes drastically for two
are more variables, where ``most'' polynomials are irreducible.

Approximations to the number of reducible multivariate polynomials go
back to Leonard Carlitz and Stephen Cohen in the 1960s. This question
was recently taken up by Bodin \cite{Bodin08} 
and Hou and Mullen \cite{HoMu09}. The sharpest bounds are in 
von zur Gathen \cite{Gathen08} for bivariate and von zur
Gathen {\em et al.} \cite{GaViZi10}
for multivariate polynomials.

From a geometric perspective, these results say that almost all
hypersurfaces are irreducible, and provide approximations to the
number of reducible ones, over a finite field. Can we say something
similar for other types of varieties?

This paper gives an affirmative answer for curves in $\Pp^{r}$ for
arbitrary $r$. A first question is how to parametrize the curves.
Moduli spaces only include irreducible curves, and systems of
defining equations do not work except for complete intersections.
The natural parametrization is by the Chow variety
$\mathcal{C}_{d,r}$ of curves of degree $d$ in $\mathbb{P}^{r}$, for
some fixed $d$ and $r$. The foundation of our work are the results
by Eisenbud and Harris \cite{EiHa92}, who identified the irreducible
components of $\mathcal{C}_{d,r}$ of maximal dimension. It turns out
that there is a threshold $d_{0}(r)= 4r-8$ so that for $d \geq
d_{0}(r)$, most curves are irreducible, and for $d < d_{0}(r)$, most
are reducible. This assumes $r \geq 3$; the planar case $r=2$ is
solved in the papers cited above, and single lines, with $d=1$, are
a natural exception.

Over a finite field, we obtain the following bounds for curves chosen
uniformly at random from $\mathcal{C}_{d,r}$. For $d \geq d_{0}(r)$,
Theorem \ref{coro:upper-bound-prob-red-curves} provides upper and
lower bounds on the probability that the curve is reducible over
$\mathbb{F}_{q}$. For $d \geq 6r-12$, Corollary
\ref{coro:upper-bound-prob-rel-red-curves} does so for the probability
that the curve is relatively irreducible over $\mathbb{F}_{q}$, that
is, irreducible over $\mathbb{F}_{q}$ and absolutely reducible. For
any $d$ and $r$ as above, both bounds tend to zero with growing
$q$. In fact, the rate of convergence in terms of $q$ is the same in
the upper and lower bounds, with (different) coefficients depending
only on $d$ and $r$.  Furthermore, we prove an ``average-case Weil
bound'', estimating the absolute difference between $q+1$ and the
expected number of $\mathbb{F}_{q}$-points on a curve defined over
$\mathbb{F}_{q}$.

All our estimates are explicit, without unspecified constants. The
main technical tools are B\'ezout type estimates of the degrees of
certain varieties, such as the incidence correspondence expressing
that a curve in $\mathcal{C}_{d,r}$ is contained in the variety
defined by a system of equations.

The structure of the paper is as follows. Section \ref{sec:notions}
introduces basic notations and facts, mainly concerning the B\'ezout
inequality and Chow varieties. Section \ref{sec:codimension}
determines the codimension of the set of reducible curves, for $d \geq
d_{0}(r)$. This is mainly based on \cite{EiHa92}. Section
\ref{sec:probability-curve-reducible} bounds, in several steps, the
degree of the Chow variety. These estimates form the technical core of
this paper. Section \ref{sec:RedCyc} draws the conclusions for the
probability of having a reducible curve, and Section \ref{sec:AbsRed}
applies our technology to relatively irreducible curves. The final
Section \ref{sec:Weil} yields an average Weil estimate.
%
%
\section{Notions and notations}\label{sec:notions}
Let $\fq$ be a finite field of $q=p^m$ elements, where $p$ is a
prime number, let $\cfq$ be an algebraic closure, and let
$\Pp^r\defineAs \Pp^r(\cfq)$ denote the $r$-dimensional projective
space over $\cfq$. Let $\Pp^{r*}$ denote the dual projective space
of $\Pp^r$, that is, $\Pp^{r*}\defineAs \Pp((\cfq{\!}^{r+1})^*)$.
Let $\G(k,r)$ denote the Grassmanian of $k$-dimensional linear
spaces ($k$-planes for short) in $\Pp^r$. We shall also denote by
$\A^{r+1}\defineAs \A^{r+1}(\cfq)$ the affine $(r+1)$--dimensional
space.

Let $\K$ be a subfield of $\cfq$ containing $\fq$, and let
${\K}[X_0\klk X_r]$ denote the ring of $(r+1)$--variate polynomials
in indeterminates $X_0\klk X_r$ and coefficients in $\K$. Let $V$ be
a $\K$--definable projective subvariety of $\Pp^r$ (a $\K$--variety
for short), namely the set of common zeros in $\Pp^r$ of a finite
set of homogeneous polynomials of $\K[X_0\klk X_r]$. For homogeneous
polynomials $f_1\klk f_s\in\K[X_0\klk X_r]$, we shall use the
notations $V(f_1,\dots,f_s)$ or $\{f_1=0,\dots,f_s=0\}$ to denote
the $\K$--variety $V$ defined by $f_1,\dots,f_s$. We shall denote by
$I(V)\subset {\K}[X_0\klk X_r]$ its defining ideal and by ${\K}[V]$
its coordinate ring, namely the quotient ring ${\K}[V]\defineAs
{\K}[X_0\klk X_r]/I(V)$. For any $d\ge 0$ we shall denote by
$({\K}[V])_d$ the $d$th graded homogeneous piece of the grading of
the coordinate ring ${\K}[V]$ induced by the canonical grading of
${\K}[X_0\klk X_r]$.
%
%
\subsection{Degree and B\'ezout type inequalities}
For an irreducible variety $V\subset \Pp^r$, we define its {\em
degree} $\deg V$ as the maximum number of points lying in the
intersection of $V$ with a linear variety $L\subset \Pp^r$ of
codimension $\dim V$ for which $\#(V\cap L)$ is finite. More
generally, if $V=C_1\cup\cdots\cup C_N$ is the decomposition of $V$
into irreducible components, we define the degree of $V$ as $\deg
V\defineAs \sum_{i=1}^N\deg C_i$ (cf. \cite{Heintz83}).

An important tool for our estimates is the 
\emph{B\'ezout
  inequality} (see \cite{Heintz83}, \cite{Fulton84},
\cite{Vogel84}, \cite{Catanese92}): if $V$ and $W$ are subvarieties
of $\Pp^r$, then the following inequality holds:
\begin{equation}
  \label{eq:Bezout}
  \deg (V\cap W)\le \deg V \cdot \deg W.
\end{equation}
The following inequality of B\'ezout type will also be useful (see
\cite[Proposition 2.3]{HeSc82}): if $V_1\klk V_s$ are subvarieties
of $\Pp^r$, then
\begin{equation}\label{eq:inequality-heintz-schnorr}
  \deg(V_1\cap\cdots\cap V_s)\le \deg V_1\,\big(\max_{2\le i\le s}\deg
  V_i\big){}^{\dim V_1}\ .
\end{equation}
We mention another variant of (\ref{eq:inequality-heintz-schnorr})
and \cite[Lemma 1.28]{Catanese92}, adapted to our purposes.
\begin{lemma}\label{lemma:desigualdad-bezout-andreotti}
  Let $U$ be an open subset of $\Pp^r$, let $V\defineAs V(f_1\klk
  f_m)$ be a subvariety of $\Pp^r$ defined by homogeneous polynomials
  of degree $d$ and let $V_s$ denote the union of the irreducible
  components of $V$ of codimension at most $s$, then
  \begin{equation}\label{eq:inequality-bezout-andreotti}
    \deg(U\cap V_s)\le d^s.
  \end{equation}
\end{lemma}
\begin{proof}
  Fix arbitrarily a point $\boldsymbol{x}\in \Pp^r\setminus V$. Then we may choose $a_{1,1}\klk a_{1,m}\in\cfq$ with
  $\sum_{j=1}^ma_{1,j}f_j(\boldsymbol{x})\not=0$. Setting $g_1\defineAs
  \sum_{j=1}^ma_{1,j}f_j$, we have that $\{g_1=0\}$ is an
  equidimensional projective variety of dimension $r-1$ containing
  $V$.

  Consider now the decomposition $\{g_1=0\}=\cup_{i=1}^nC_i$ of
  $\{g_1=0\}$ into irreducible components. Suppose that ${C}_i$ is not
  contained in $V$ for $1\le i\le n_1$ and it is contained in $V$ (and
  hence it is a component of $V$) for $n_1+1\le i\le n$. Then there
  exist $\boldsymbol{x}^{(2,i)}\in {C}_i\setminus V$ for $1\le i\le n_1$, and
  $a_{2,1}\klk a_{2,m}\in\cfq$ such that no point $\boldsymbol{x}^{(2,i)}$ is a
  zero of the polynomial $g_2\defineAs
  \sum_{j=1}^ma_{2,j}f_j$. Observe that $\{g_1=0,g_2=0\}$ contains all
  the components of $V$ of codimension at most 2 among its irreducible
  components.

  Arguing inductively, we see that there exist homogeneous
  polynomials $g_1\klk g_{s}\in\cfq[X_0\klk X_r]$ of degree $d$ with
  the following property: all the irreducible components of $V$ of
  codimension at most $s$ are irreducible components of $\{g_1=0\klk
  g_s=0\}$. In particular, all the irreducible components of $V_s$ are
  irreducible components of $\{g_1=0\klk g_s=0\}$. By the definition
  of degree and the B\'ezout inequality
  (\ref{eq:Bezout}) it follows that $\deg(U\cap V_s)\le d^s$, which
  finishes the proof of the lemma.
\qed\end{proof}
Finally, we shall also use the following well-known inequality, which
is proved here for lack of a suitable reference.
\begin{lemma}\label{lemma:bound-deg-morphism}
  Let $V$ be a projective subvariety of $\Pp^r$ and let $F\defineAs
  (f_0\klk f_s):V\to\Pp^s$ be a regular mapping defined by homogeneous
  polynomials of degree $d$. If $m$ denotes the dimension of $F(V)$,
  then $\deg F(V)\le \deg V\cdot d^m$.
\end{lemma}
\begin{proof}
  We may assume without loss of generality that $V$ is irreducible.
  Then $F(V)$ is an irreducible variety of $\Pp^s$. Let $H_1\klk H_m$
  be hyperplanes of $\Pp^s$ such that $\#(F(V)\cap H_1\cap\cdots\cap
  H_m)=\deg F(V)$ holds. Let $\mathcal{S}\defineAs F(V)\cap
  H_1\cap\cdots\cap H_m$. Then
  $$F^{-1}(\mathcal{S})=V\cap F^{-1}(H_1)\cap\cdots\cap F^{-1}(H_m).$$

  Observe that $F^{-1}(H_i)=V\cap\{g_i=0\}$, where $g_i$ is a linear
  combination of the polynomials $f_0\klk f_s$ for $i=1\klk m$.
  Therefore, by the B\'ezout inequality (\ref{eq:Bezout}) it follows
  that $\deg F^{-1}(\mathcal{S})\le \deg V\cdot d^m$ holds. Let
  $F^{-1}(\mathcal{S})=\cup_{i=1}^NC_i$ be the decomposition of
  $F^{-1}(V)$ into irreducible components. Since
  $F(F^{-1}(\mathcal{S}))=\mathcal{S}$ and each irreducible component
  $C_i$ of $F^{-1}(\mathcal{S})$ is mapped by $F$ to a point of
  $\mathcal{S}$, we have
  $$\deg F(V)=\#(\mathcal{S})\le N\le\sum_{i=1}^N\deg C_i=\deg
  F^{-1}(\mathcal{S})\le \deg V\cdot d^m.$$
  This finishes the proof of the lemma.
\qed\end{proof}
%
%
%
\subsection{$\fq$-rational points}
The set of $\fq$-rational points of $V$, namely $V\cap \Pp^r(\fq)$,
is denoted by $V(\fq)$. In some simple cases it is possible to
determine the exact value of $\# V(\fq)$. For instance, the number
$p_r$ of elements of $\Pp^r(\fq)$ is given by $p_r=q^r+q^{r-1}\plp
q+1$. In what follows we shall use repeatedly the following
elementary upper bound on the number of $\fq$-rational points of a
projective variety $V$ of dimension $s$ and degree $d$ (see, e.g.,
\cite[Proposition 12.1]{GhLa02a} or \cite[Proposition 3.1]{CaMa07}):
\begin{equation}\label{eq:upper-bound-number-rat-points}
  \# V(\fq) \le d p_s \leq 2dq^{s}.
\end{equation}
%
%
\subsection{Chow varieties of curves}
Suppose that $r\ge 2$ and fix $d>0$. Consider the incidence variety
$$\Psi\defineAs \{(\boldsymbol{x},H_1,H_2)\in\Pp^r\times
\Pp^{r*}\times\Pp^{r*}:\boldsymbol{x}\in H_1\cap H_2\}.$$
Let $\pi:\Psi\to\Pp^r$ and $\eta:\Psi\to\Pp^{r*}\times\Pp^{r*}$ denote
the standard projections. Fix a curve $C\subset\Pp^r$ of degree $d$
and consider the ``restricted'' incidence variety
$$\Psi_C\defineAs \pi^{-1}(C)\defineAs \{(\boldsymbol{x},H_{1}, H_{2})\in
C\times\Pp^{r*}\times\Pp^{r*}:\boldsymbol{x}\in H_1\cap H_2\}.$$
It turns out that $\eta(\Psi_C)$ is a bihomogeneous hypersurface of
$\Pp^{r*}\times\Pp^{r*}$ of bidegree $(d,d)$ (see, e.g.,
\cite[Lecture 21]{Harris92}). This hypersurface is thus defined by a
reduced bihomogeneous polynomial
$F_C\in\cfq[\boldsymbol{A},\boldsymbol{B}]\defineAs \cfq[A_0\klk
A_r,B_0\klk B_r]$ of bidegree $(d,d)$, unique up to scaling by
nonzero elements of $\cfq$. In this way, we see that $\eta(\Psi_C)$
can be represented by a point $[F_C]$ in the projective space
$\Pp\V_{d,r}$, where $\V_{d,r}$ denotes the vector subspace of
$\cfq[\boldsymbol{A},\boldsymbol{B}]$ spanned by all the
bihomogeneous polynomials of bidegree $(d,d)$. The point $[F_C]$ is
called the \emph{Chow form} of $C$. The Zariski closure in
$\Pp\V_{d,r}$ of the set of points $[F_C]$, where $C$ runs over the
set of curves (equidimensional varieties of dimension 1) of degree
$d$ of $\Pp^r$, is called the \emph{Chow variety} of curves of
degree $d$ in $\Pp^r$ and denoted by $\mathcal{C}_{d,r}$.
%
%
\subsubsection{The correspondence between degree-$d$ cycles in $\Pp^r$
  and points in $\Pp\V_{d,r}$} \label{subsubsec:correspondence curves
  - points}
Each point of the Chow variety $\mathcal{C}_{d,r}$ actually
corresponds to a unique \emph{effective cycle} on $\Pp^r$ of dimension
1 and degree $d$, that is, to a formal linear combination $\sum
a_iC_i$, where each $C_i$ is an irreducible curve of $\Pp^r$, each
$a_i$ is a positive integer and $\sum a_i\deg(C_i)=d$.  Such a
correspondence is defined assigning to each cycle $\sum a_iC_i$ the
point of $\Pp\V_{d,r}$ determined by the polynomial $\prod_i
F_{C_i}^{a_i}$, where $F_{C_i}$ is a minimal--degree defining
polynomial of the hypersurface $\eta(\Psi_{C_i})$ for each $i$.

Let $\sum a_iC_i$ be an effective cycle of dimension 1 and degree
$d$ and let $[F]\in\mathcal{C}_{d,r}$ be the corresponding Chow
form. Let $\{f_\lambda:\lambda\in\Lambda\}$ be the set of all
nonzero coefficients of $F$. Following, e.g., \cite[Exercise
I.1.18]{Kollar99}, we define the \emph{smallest field of definition}
of $[F]$ as the field extension $\K$ of $\fp$ determined by the
fractions of nonzero coefficients of $F$, namely $\K\defineAs
\fp(f_\lambda/f_\mu:\lambda,\mu\in\Lambda)$. It follows that there
exists a scalar multiple of $F$ with coefficients in $\K$, and there
are no scalar multiples of $F$ with coefficients in a proper
subfield of $\K$. For an arbitrary subfield $\K$ of $\cfq$, we say
that an effective cycle $\sum a_iC_i$ is $\K$-definable (a
$\K$-cycle for short) if the smallest field of definition of its
Chow form $[\prod_iF_{C_i}^{a_i}]$ is a subfield of $\K$. We use the
following result.
\begin{theorem}[{\cite[Corollary I.3.24.5]{Kollar99}}]
  \label{th:correspondence-chow-points-and-curves} Let $\K$ be a
  subfield of $\cfq$. There exists a one-to-one correspondence between
  the set of effective $\K$-cycles of dimension 1 and degree $d$ and
  the set of $\K$-rational points in the Chow variety
  $\mathcal{C}_{d,r}$.
\end{theorem}
The inverse of the correspondence of Theorem
\ref{th:correspondence-chow-points-and-curves} can be explicitly
described in the following terms. Let $[F]$ be a point of
$\mathcal{C}_{d,r}$ with $F$ reduced. We define the \emph{support} of
$[F]$ by
\begin{equation}\label{eq:support-of-a-chow-form}
  \mathrm{supp}(F)\defineAs \{\boldsymbol{x}\in\Pp^r: \pi^{-1}(\boldsymbol{x})\subset\eta^{-1}(V_F)\},
\end{equation}
where $V_F$ is the hypersurface of $\Pp\V_{d,r}$ defined by $F$.  We
have $\mathrm{supp}(F_C)=C$ for every curve $C\subset\Pp^r$ of degree
$d$ (see, e.g., \cite[Lecture 21]{Harris92}). This identity is
extended straightforwardly to cycles by taking into account the
multiplicity of each factor of $F$.
%
%
\subsubsection{Reducible and $\K$-reducible cycles}
\label{subsubsec:reducible-cycles}
If $\K$ is a subfield of $\cfq$, an effective $\K$-cycle $C$ is called
\emph{$\K$-reducible} if there exist $s\ge 2$ and effective $\K$-cycles
$C_1\klk C_s$ such that $C=\sum_{i=1}^sC_i$ holds.  When $\K=\cfq$ we
shall omit the reference to the field of definition and simply speak
about \emph{(absolutely) reducible} effective cycles.

The set $\mathcal{R}_{d,r}$ of reducible effective cycles of $\Pp^r$
of dimension 1 and degree $d$ is a closed subset of the Chow variety
$\mathcal{C}_{d,r}$. In order to prove this, fix $1\le k\le d-1$ and
consider the Chow varieties $\mathcal{C}_{k,r}$ and
$\mathcal{C}_{d-k,r}$. We have a regular map
$$
\mu_{k,d,r}:\mathcal{C}_{k,r}\times\mathcal{C}_{d-k,r}\to
\mathcal{C}_{d,r},
$$
which is induced by the multiplication mapping
$\Pp\V_{k,r}\times\Pp\V_{d-k,r}\to\Pp\V_{d,r}$.  It is easy to see
that the following identity holds:
\begin{equation}\label{eq:reducible-curves-1}
  \mathcal{R}_{d,r}=\bigcup_{1\le
    k\le  d/2 } \mathrm{im}(\mu_{k,d,r}).
\end{equation}
Since $\mathcal{C}_{k,r}\times\mathcal{C}_{d-k,r}$ is a projective
variety and the image of a projective variety under a regular map is
closed we conclude that $\mathcal{R}_{d,r}$ is a closed subset of
$\mathcal{C}_{d,r}$.
%
%
\section{The codimension of the variety of reducible
  curves}\label{sec:codimension}
The irreducibility of a single polynomial over a finite field shows
a qualitatively different behavior between one and at least two
variables. In the former case, fairly few (a fraction of about $1/d$
at degree $d \geq 2$) are irreducible, while in the latter case
almost all (a fraction of about $1-q^{-d+1}$ over $\mathbb{F}_{q}$)
are irreducible. One may wonder whether such a qualitative jump also
occurs for systems of more than one polynomial. We consider this
question in a special case, namely where the equations define a
curve in $\mathbb{P}^{r}$. It turns out that there is a threshold
$d_{0}(r)=4r-8$ for the degree where this jump occurs. At lower
degrees, curves are generically reducible (with single lines, of
degree $1$, as a natural exception), while at degree $d \geq
d_{0}(r)$ a generic curve is irreducible.

The qualitative result follows from the work of Eisenbud and Harris
\cite{EiHa92}. Our contribution are quantitative estimates for the
fractions under consideration. Fairly precise bounds are available
for single polynomials (that is, planar curves). Similarly, our
lower and upper bounds for $d \geq d_{0}(r)$ are given by the same
power of $q$, but with two different coefficients depending on $d$
and $r$.

The foundation for our work are the following results from
\cite{EiHa92}. They consider two irreducible subvarieties of
$\mathcal{C}_{d,r}$:
\begin{align*}
  dG(1,r)& \defineAs  \{\text{sums of}~ d ~\text{lines in}~ \mathbb{P}^{r}\}, \nonumber \\
  P(d,r)& \defineAs \{\text{plane curves of degree $d$ in}~
  \mathbb{P}^{r}\}. \nonumber
\end{align*}
It is easy to see that, for $d \geq 2$,
\begin{align*}
  \dim dG(1,r)& = 2d(r-1),\\
  \dim P(d,r) & = 3(r-2)+ d(d+3)/2.
\end{align*}
\begin{fact}[{Eisenbud \&\ Harris \cite[Theorems 1 and 3]{EiHa92}}]
  \label{fact:dimension-chow-var-eis-harris}
  Let $d \geq 2$ and $r \geq 3$. Then
  \begin{equation}\label{eq:dimension-chow-var-for-r-ge-4}
    \dim \mathcal{C}_{d,r}= \max ~\{2d(r-1), 3(r-2)+ d(d+3)/2\}.
  \end{equation}
  For $r=3$ and $d \geq4$, and for $r \geq4$, either $dG(1,r)$ or
  $P(d,r)$ is the unique irreducible component of maximal dimension.
\end{fact}
We let $\mathcal{R}_{d,r}\defineAs \{ C \in \mathcal{C}_{d,r} \colon
\text{$C$ reducible} \}$.  In the case $(d,r)= (2,3)$, both
$\mathcal{R}_{2,3}=2G(1,3)$ and $P(2,3)$ have dimension $8$, so that
$\mathrm{codim}_{\mathcal{C}_{d,r}}\mathcal{R}_{d,r}=0$.  In the case
$(d,r) = (3,3)$, we have $\mathcal{R}_{3,3}\supset 3G(1,3)$ and both
$3G(1,3)$ and $P(3,3)$ have dimension $12$, so that
$\mathrm{codim}_{\mathcal{C}_{d,r}}\mathcal{R}_{d,r}=0$.  We have
\begin{align*}
  2d(r-1)\leq 3(r-2)+ \frac{d(d+3)}{2} \Longleftrightarrow 4r -
  \frac{17}{2} - \frac{3}{2(2d-3)} \leq d,
\end{align*}
and also equalities in the two conditions correspond to each
other. Since $3/(2d-3)$ is not an integer for $d>3$, $dG(1,r)$ and
$P(d,r)$ never have the same dimension except for $d=1$, where $P(1,r)
= 1G(1,r)= G(1,r)$ has dimension $2(r-1)$, and for the exceptional
cases from above, where
\begin{equation}\label{eq:d,r}
  (d,r) \text{ is } (2,3) \text{ or } (3,3).
\end{equation}
Furthermore, $3/2(2d-3)< 1/2$ for $d>3$. We abbreviate $b_{d,r} = \dim
\mathcal{C}_{d,r}$ and reword Fact
\ref{fact:dimension-chow-var-eis-harris} as follows.

\begin{fact}\label{fact:dimension-chow-var}
  Let $d \geq 2$ and $r\geq3$. Then
  \begin{displaymath}
    b_{d,r}= \dim \mathcal{C}_{d,r}= \left\{\begin{array}{ll}
        3(r-2)+ {d(d+3)}/{2} & \text{if }d\geq 4r-8,\\
        2d(r-1) & \text{otherwise}.\\
      \end{array}\right.
  \end{displaymath}
  $\mathcal{C}_{d,r}$ has exactly one component of maximal dimension,
  namely $P(d,r)$ and $dG(1,r)$ in the first and second case,
  respectively, except for (\ref{eq:d,r}).
\end{fact}

When $d<4r-8$ and excepting (\ref{eq:d,r}), then $dG(1,r)$ is the
dominant component of $\mathcal{C}_{d,r}$ and a generic curve in
$\mathcal{C}_{d,r}$ is reducible. On the other hand, for $d\ge 4r-8$
the generic curve is irreducible, and we now want to determine the
codimension of the reducible ones. For planar curves and $d \geq 2$
this codimension is $d-1$, and the dominating component in the
reducible ones consists of curves that are a union of a line and an
irreducible (planar) curve of degree $d-1$ (see \cite[Theorem
2.1]{Gathen08}). For $r \geq 3$, the dimension of this set of curves
is $2(r-1)+3(r-2)+ {(d-1)(d+2)}/{2}$ when $d \geq 4r-7$, and the
codimension is $d-2r+3$.

Fix $1\le k\le d-1$ and consider the Chow varieties
$\mathcal{C}_{k,r}$ and $\mathcal{C}_{d-k,r}$. Recall the morphism
$$
\mu_{k,d,r}:\mathcal{C}_{k,r}\times\mathcal{C}_{d-k,r}\to
\mathcal{C}_{d,r},
$$
induced by the multiplication mapping
$\Pp\V_{k,r}\times\Pp\V_{d-k,r}\to\Pp\V_{d,r}$.
Our aim is to bound in (\ref{eq:reducible-curves-1}) the dimension
of the image $\mathrm{im}(\mu_{k,d,r})$ of $\mu_{k,d,r}$.

\begin{theorem}\label{th:dimension-variety-red-curves}
  Let $r \geq 3$ and $d \geq 4r-8$. Then
  \begin{displaymath}
    \mathrm{codim}_{\mathcal{C}_{d,r}}\mathcal{R}_{d,r}= \left\{\begin{array}{ll}
        r-2 & ~\text{if} ~d=4r-8,\\
        d-2r+3 & ~\text{otherwise},
      \end{array}\right.
  \end{displaymath}
  \begin{displaymath}
    \mathrm{dim} ~ \mathcal{R}_{d,r}= \left\{\begin{array}{ll}
        8(r-1)(r-2) & ~\text{if} ~d=4r-8,\\
        5r-9+d(d+1)/2 & ~\text{otherwise}.
      \end{array}\right.
  \end{displaymath}
\end{theorem}

\begin{proof}
  We let $K\defineAs \{1,\ldots,\lfloor d/2\rfloor \}$,
  $b_{d,r}\defineAs \dim\mathcal{C}_{d,r}$ for $d>0$, and
  $$
  u(k)\defineAs b_{d,r}-b_{k,r}-b_{d-k,r}
  $$
  for $k \in K$, so that $\mathrm{codim}_{\mathcal{C}_{d,r}}
  \mathcal{R}_{d,r} \geq \min\{u(k)\colon k \in K\}$. We abbreviate
  the latter as $m$ and first show that it equals the value claimed
  for the codimension.

  We note that $d/2 \geq 2$ and $k \leq d-k$ for all $k \in K$ and
  define a partition of $K$ into three subsets $K_{1}, K_{2}, K_{3}$
  as follows:
  \begin{align*}
    &K_{1}= \{k \in K \colon k \geq4r-8\},\\
    &K_{2}= \{k \in K \colon k<4r-8,d-k \geq 4r-8\},\\
    &K_{3} = \{k \in K \colon d-k < 4r-8\}.
  \end{align*}
  Furthermore, we let
  $$
  m_{i}= \min \{u(k)\colon k\in K_{i}\}
  $$
  for $1 \leq i \leq 3$, with $\min \varnothing = \infty$. Then $m=
  \min \{m_{1}, m_{2}, m_{3}\}$ and according to Fact
  \ref{fact:dimension-chow-var},
  \begin{displaymath}
    u(k) = \left\{\begin{array}{lll}
        k(d-k)-3(r-2) & ~\text{if}~ k \in K_{1}, \\
        ({k}/{2})(2d-4r+7-k) & ~\text{if} ~k\in K_{2},\\
        3(r-2)-2dr+{d(d+7)}/{2}& ~\text{if}~ k\in K_{3}.
      \end{array}\right.
  \end{displaymath}
  In the last line, $u(k)$ does not depend on $k$.

  If $d = 4r-8$, then $K=K_{3}$ and
  $$
  u(k)= u(1)= r-2=m_{3}
  $$
  for all $k \in K$.

  We may now assume that $d \geq 4r-7$. Then $1 \in K_{2}$ and $u(1)=
  d-2r+3$. For $k \in K_{2}$, $ u(k)$ is a quadratic function of $k$
  with roots $0$ and $2d-4r+7$ and takes its minimum in the range
  $1\leq k\leq 2d-4r+6$ at $k=1$. Since $k \leq d-4r+8 \leq 2d-4r+6$
  for all $k\in K_{2}$, the range includes all of $K_{2}$. Thus
  $m_{2}=d-2r+3$.

  To determine $m_{1}$, we may assume that $d/2 \geq 4r-8$, since
  otherwise $K_{1}= \varnothing$. The quadratic function $k(d-k)$
  takes its minimum value in $K_{1}$ at $4r-8$. Now
  \begin{align*}
    & m_{1}= -16r^{2}+4dr+61r-8d-58=u(4r-8)\geq u(1)=m_{2}
    \\
    & \Longleftrightarrow d\geq 4r-\frac{27}{4}+\frac{1}{4(4r-9)}.
  \end{align*}
  The last inequality is strictly satisfied, since
  $$
  d\geq 8r-16>4r-\frac{27}{4}+\frac{1}{4(4r-9)}.
  $$
  Thus $m_{1}>m_{2}$.  For $k \in K_{3}$, we have
  $$
  m_{3}= u(k)\geq u(1)=m_{2} \Longleftrightarrow 4r-\frac{15}{2} -
  \frac{3}{2(2d-5)}\leq d.
  $$
  The last condition is strictly satisfied, and therefore $m_{3}>
  m_{2}$ and $m=m_{2}$.  In all cases, we have $m=u(1)$.

  In order to prove a lower bound on $\dim \mathcal{R}_{d,r}$, it is
  sufficient to show that $\mu_{1,d,r}$ has some finite fiber, since
  then
  $$
  \text{codim}_{\mathcal{C}_{d,r}}(\text{im}\,\mu_{1,d,r})\leq m.
  $$
  If $d \geq 4r-7$, then $1 \in K_{2}$, $d-1 \geq4r-8$,
  $\text{codim}_{\mathcal{C}_{d-1,r}}\mathcal{R}_{d-1,r}>0$, and most
  curves in $\mathcal{C}_{d-1,r}$ are irreducible. Thus $\mu_{1,d,r}$
  restricted to $\mathcal{C}_{1,r} \times (\mathcal{C}_{d-1,r}
  \setminus \mathcal{R}_{d-1,r})$ is injective, and in particular we
  have some finite fiber. If $d=4r-8$, then $dG(1,r)\subseteq
  \mathcal{R}_{d,r} $ and $\text{codim}_{\mathcal{C}_{d,r}}
  \mathcal{R}_{d,r} \leq \text{codim}_{ \mathcal{C}_{d,r}}dG(1,r)
  =r-2$.
The claims about $\text{dim} \mathcal{R}_{d,r}$ follow from Fact
\ref{fact:dimension-chow-var}. \qed\end{proof}
%

%
\section{The degree of the Chow variety of curves}
\label{sec:probability-curve-reducible}

Recall that an effective $\fq$-cycle $C$ of $\Pp^r$ of dimension 1
is $\fq$-reducible if there exist $s\ge 2$ and effective
$\fq$-cycles $C_1\klk C_s$ such that $C=\sum_{i=1}^sC_i$.  According
to Theorem \ref{th:correspondence-chow-points-and-curves}, to each
$\fq$-cycle of degree $d$ corresponds a unique $\fq$-rational point
of $\mathcal{C}_{d,r}$. Therefore, to each $\fq$-cycle of degree $d$
which is $\fq$-reducible corresponds at least one point in the image
of $\mathcal{C}_{k,r}(\fq) \times \mathcal{C}_{d-k,r}(\fq)$ under
the mapping $\mu_{k,d,r}$ for a given $k\in K\defineAs
\{1\klk\lfloor{d}/{2}\rfloor\}$. More precisely, we have
$$\mathcal{R}_{d,r}^{(q)}(\fq)\defineAs \{C\in \mathcal{R}_{d,r}(\fq):
C\mbox{ is $\fq$-reducible}\}=\bigcup_{1\le k\le d/2}
\mu_{k,d,r}\big(\mathcal{C}_{k,r}(\fq) \times
\mathcal{C}_{d-k,r}(\fq)\big).$$
From (\ref{eq:upper-bound-number-rat-points}) we conclude that
\begin{gather*}
  \begin{aligned}
    \# \mathcal{R}_{d,r}^{(q)}(\fq)
    &\le \sum_{1\le k\le d/2} \# \mathcal{C}_{k,r}(\fq) \cdot \#
    \mathcal{C}_{d-k,r}(\fq)
    \\
    &\le \sum_{1\le k\le d/2} \deg\mathcal{C}_{k,r} \cdot 2q^{b_{k,r}}
    \cdot \deg\mathcal{C}_{d-k,r} \cdot 2q^{b_{d-k,r}},
  \end{aligned}\end{gather*}
with $b_{d,r} = \dim \mathcal{C}_{d,r}$. If $d \geq 4r-8 \geq 4$, then
Theorem \ref{th:dimension-variety-red-curves} implies that
\begin{equation}
  \# \mathcal{R}_{d,r}^{(q)}(\fq)\le\left\{\begin{array}{ll}
      4 \sum_{1\le k\le d/2}
      ~ \deg\mathcal{C}_{k,r}\deg\mathcal{C}_{d-k,r}
      \,q^{b_{d,r}-r+2}&  \text{if }\ d=4r-8, \\
      4\sum_{1\le k\le d/2}
      ~ \deg\mathcal{C}_{k,r}\deg\mathcal{C}_{d-k,r}
      \,q^{b_{d,r}-(d-2r+3)}  & \text{otherwise}.
    \end{array}\right.
  \label{eq:expression-upper-bound-red-curves}\end{equation}
%
%
\subsection{An upper bound on the degree of the restricted Chow
  variety $\widetilde{\mathcal{C}}_{d,r}$}
The inequality (\ref{eq:expression-upper-bound-red-curves}) shows
that an upper bound on the number of $\fq$-reducible cycles in
$\Pp^r$ of dimension 1 and degree $d$ can be deduced from an upper
bound on the degree of the Chow variety $\mathcal{C}_{d,r}$ of
curves over $\overline{\mathbb{F}}_{q}$ of degree $d$ in $\Pp^r$. In
order to obtain an upper bound on the latter, we consider a suitable
variant of the approach of Koll\'ar \cite[Exercise I.3.28]{Kollar99}
(see also \cite{Guerra99}).

With the terminology of \cite{EiHa92}, we shall consider the {\em
  restricted Chow variety} $\widetilde{\mathcal{C}}_{d,r}$ of curves
of degree $d$ of $\Pp^r$, namely the union of the irreducible
components of $\mathcal{C}_{d,r}$ whose generic point corresponds to
a nondegenerate absolutely irreducible curve of $\Pp^r$.

Our purpose is to obtain an upper bound on the degree of
$\widetilde{\mathcal{C}}_{d,r}$, from which an upper bound on the
degree of ${\mathcal{C}}_{d,r}$ is readily obtained.
%
%
\subsubsection{An incidence variety related to
  $\widetilde{\mathcal{C}}_{d,r}$}
Let $\Pp^N$ denote the projective space of sequences
$\boldsymbol{f}=(f_0\klk f_r)$ of homogeneous polynomials in
$\cfq[X_0\klk X_r]$ of degree $d$, and consider the incidence
variety
\begin{equation}\label{eq:definition-incidence-variety}
  \Gamma\defineAs \Gamma_{d,r}\defineAs \{(\boldsymbol{f},[F])\in\Pp^N\times
  \widetilde{\mathcal{C}}_{d,r}: V(\boldsymbol{f})\supset \mathrm{supp}(F)\}.
\end{equation}
In this section we obtain an upper bound on the degree of $\Gamma$.

Let $\theta:\Gamma\to\Pp^N$ and $\phi:\Gamma\to
\widetilde{\mathcal{C}}_{d,r}$ denote the corresponding projections.
The $\phi$-fiber of a cycle $[F]$ corresponding to a curve $C$
consists of the set of sequences $\boldsymbol{f}=(f_0\klk f_r)$
vanishing on $C$. On the other hand, it is clear that the image
$\theta(\Gamma)$ is contained in the Zariski closed subset
$\mathcal{U}$ of $\Pp^N$ defined by
$$\mathcal{U}\defineAs \{\boldsymbol{f}\in\Pp^N:\dim V(f)\ge 1\}.$$

According to \cite[Exercise I.3.28]{Kollar99}, the following facts
hold:
\begin{enumerate}
\item[(1)] $\mathcal{U}$ is a closed subset of $\Pp^N$,
\item[(2)] $\mathcal{U}$ can be defined by polynomials of degree
  $\binom{rd+d}{r}$.
\end{enumerate}

Let $\mathcal{T}$ be an absolutely irreducible component of
$\Gamma$. We have the following assertions (see \cite[Proposition
2.4]{Guerra99}):
\begin{enumerate}
\item[(3)] $\theta(\mathcal{T})$ is an absolutely irreducible
  component of $\mathcal{U}$,
\item[(4)] $\theta(\mathcal{T})$ has codimension at most
  $(r+1)(d^2+1)$ in $\Pp^N$,
\item[(5)] $\theta|_{\mathcal{T}}:\mathcal{T}\to \theta(\mathcal{T})$
  is a birational map.
\end{enumerate}
\begin{lemma}\label{lemma:estimate-degree-gamma-theta}
  We have the estimate
  $$\deg \theta(\Gamma)\le \binom{rd+d}{r}^{(r+1)(d^2+1)}.$$
  %
\end{lemma}
\begin{proof} Denote by $\mathcal{U}_{(r+1)(d^2+1)}$ the union of the
  absolutely irreducible components of $\mathcal{U}$ of codimension at
  most $(r+1)(d^2+1)$. According to (3)--(4), all the absolutely
  irreducible components of $\theta(\Gamma)$ are absolutely
  irreducible components of $\mathcal{U}_{(r+1)(d^2+1)}$.  Thus, by
  definition of degree it follows that $\deg\theta(\Gamma)
  \le\deg\mathcal{U}_{(r+1)(d^2+1)}$. By Lemma
  \ref{lemma:desigualdad-bezout-andreotti} we have
  $$\deg\mathcal{U}_{(r+1)(d^2+1)}\le \binom{rd+d}{r}^{(r+1)(d^2+1)},$$
  from which the lemma follows.
\qed\end{proof}

From the proof of \cite[Proposition 2.4]{Guerra99} one deduces that
the points $\boldsymbol{f}\in \theta(\Gamma)$ for which
$V(\boldsymbol{f})$ is an irreducible curve of $\Pp^r$ of degree $d$
form a dense open subset of $\theta(\Gamma)$. Let $\mathcal{V}$ be
the dense open subset of $\theta(\Gamma)$ where the inverse mapping
$\theta^{-1}$ of $\theta$ is well-defined and $V(\boldsymbol{f})$ is
an irreducible curve of degree $d$ for every $\boldsymbol{f}\in
\mathcal{V}$. By the definition of $\mathcal{V}$ it turns out that
$\deg\theta(\Gamma)= \deg\mathcal{V}$ and $\theta^{-1}(\mathcal{V})$
is a dense open subset of $\Gamma$, which in turn implies the
equality $\deg\theta^{-1}(\mathcal{V})=\deg\Gamma$. Furthermore, we
have the identity
\begin{equation}\label{eq:def-prelim-open-subset-incidence}
  \theta^{-1}(\mathcal{V})=\{(\boldsymbol{f},[F])\in\mathcal{V}\times\mathcal{I}_{d,r}:
  V(\boldsymbol{f})\supset\mathrm{supp}(F)\},
\end{equation}
where $\mathcal{I}_{d,r}$ denotes the set of nondegenerate irreducible
curves of $\Pp^r$ of degree $d$.

Denote by $(\cfq[\G])_d\defineAs (\cfq[\G(r-2,r)])_d$ the
$d$--graded piece of the coordinate ring of the Grassmannian
$\G\defineAs \G(r-2,r)$.  Arguing as in Section
\ref{subsubsec:reducible-cycles} we see that the set
$\mathcal{I}((\cfq[\G])_d)$ of irreducible elements of
$(\cfq[\G])_d$ is an open subset of $(\cfq[\G])_d$. If we denote by
$\mathcal{I}_{\ge 1}((\cfq[\G])_d)$ the set of irreducible cycles of
dimension 1 and degree $d$, taking into account that each
$\boldsymbol{f}\in\mathcal{V}$ defines an irreducible curve
$V(\boldsymbol{f})$ of degree $d$, from
(\ref{eq:def-prelim-open-subset-incidence}) and \cite[Corollary
I.3.24.5]{Kollar99} we deduce the identity:
\begin{equation}\label{eq:def-open-subset-incidence}
  \theta^{-1}(\mathcal{V})=\{(\boldsymbol{f},[F])\in\mathcal{V}\times
  \mathcal{I}_{\ge 1}((\cfq[\G])_d):
  V(\boldsymbol{f})\supset\mathrm{supp}(F)\}.
\end{equation}
\begin{proposition}\label{prop:degree-eqs-incidence-var}
  $\theta^{-1}(\mathcal{V})$ is defined in the product
  $\mathcal{V}\times \mathcal{I}_{\ge 1}((\cfq[\G])_d)$ by
  bihomogeneous polynomials of bidegree at most $(d,D_r)$, where
  $D_r\defineAs \binom{d^2+r}{r}$.
\end{proposition}
\begin{proof} Let $\boldsymbol{f}\in\mathcal{V}$, let $C\defineAs V(\boldsymbol{f})$ and let
  $[F]$ be an element of $\mathcal{I}((\cfq[\G])_d)$. Then
  \cite[Corollary I.3.24.5]{Kollar99} shows that $\mathrm{supp}(F)$ is
  an irreducible curve of $\Pp^r$ of degree $d$. Assume without loss
  of generality that $\mathrm{supp}(F)$ has no components contained in
  the hyperplane $\{X_0=0\}$ at infinity. By
  (\ref{eq:support-of-a-chow-form}) it follows that
  $$\mathrm{supp}(F)=\{\boldsymbol{x}\in\Pp^r:\pi^{-1}(\boldsymbol{x})\subset\eta^{-1}(V_{F})\}.$$
  Let $A_0\klk A_r,\!B_0\klk\!B_r$ be new indeterminates, let $\boldsymbol{A}\defineAs
  (A_0\klk A_r)$, $\boldsymbol{B}\defineAs (B_0\klk B_r)$ and $\boldsymbol{X}\defineAs (X_0\klk
  X_r)$, and let $F\defineAs
  \sum_{|\boldsymbol{\alpha}|=|\boldsymbol{\beta}|=d}a_{\boldsymbol{\alpha},\boldsymbol{\beta}}
  \boldsymbol{A}^{\boldsymbol{\alpha}} \boldsymbol{B}^{\boldsymbol{\beta}}$.  Then
  $$\pi^{-1}(\boldsymbol{x})=\{(\boldsymbol{x},H_1,H_2):A_0x_0\plp A_rx_r=B_0x_0\plp
  B_rx_r=0\}.$$
  Since $\mathrm{supp}(F)$ has no components at infinity, we see that
  the condition $\pi^{-1}(\boldsymbol{x})\subset\eta^{-1}(V_{F})$ is equivalent to
  the identity
  \begin{equation}\label{eq:equations-for-a-chow-form}
    F\left(-\sum_{i=1}^rA_iX_i,A_1X_0\klk A_rX_0,
    -\sum_{i=1}^rB_iX_i,B_1X_0\klk B_rX_0\right)=0.
  \end{equation}
  Denote by $\widehat{F}(\boldsymbol{A},\boldsymbol{B},\boldsymbol{X})$ the polynomial in the left-hand side
  of the previous expression and write
  $$\widehat{F}(\boldsymbol{A},\boldsymbol{B},\boldsymbol{X})
  \defineAs \!\sum_{\boldsymbol{\gamma},\boldsymbol{\delta}}
  c_{\boldsymbol{\gamma},\boldsymbol{\delta}}(\boldsymbol{a},\!\boldsymbol{X})
  \boldsymbol{A}^{\boldsymbol{\gamma}}\boldsymbol{B}^{\boldsymbol{\delta}},$$
  where $\boldsymbol{a}\defineAs (a_{\boldsymbol{\alpha},
  \boldsymbol{\beta}})_{\boldsymbol{\alpha},\boldsymbol{\beta}}$ is the vector
  of coefficients of $F$ and $\boldsymbol{\gamma},\boldsymbol{\delta}$ run through the set of
  elements $\boldsymbol{\nu}\in(\Z_{\ge 0})^r$ with $|\boldsymbol{\nu}|=d$. Then
  (\ref{eq:equations-for-a-chow-form}) is equivalent to the following
  defining system of $\mathrm{supp}(F)$ in $\Pp^r$:
  \begin{equation}\label{eq:equations-for-a-chow-form-bis}
    \{c_{\boldsymbol{\gamma},\boldsymbol{\delta}}(\boldsymbol{a},\boldsymbol{X})=0:
    |\boldsymbol{\gamma}|=|\boldsymbol{\delta}|=d\}.
  \end{equation}
  Observe that each polynomial $c_{\boldsymbol{\gamma},\boldsymbol{\delta}}(\boldsymbol{a},\boldsymbol{X})$ is
  bihomogeneous of degree 1 in $\boldsymbol{a}$ and degree $2d$ in $\boldsymbol{X}$.

  If the inclusion $V(\boldsymbol{f})\supset \mathrm{supp}(F)$ is fulfilled then
  the ideal $(f_0\klk f_r)$ generated by $f_0\klk f_r$ is included in
  the ideal $I(C)$ of the curve $C\defineAs \mathrm{supp}(F)$. The
  latter condition implies the corresponding inclusion $(f_0\klk
  f_r)^d\subset I(C)^d$ of $d$th powers of both ideals.  According to
  \cite{Amoroso94} or \cite{Kollar99a}, the inclusion of ideals
  $I(C)^d\subset(c_{\boldsymbol{\gamma},\boldsymbol{\delta}}:
  \boldsymbol{\gamma},\boldsymbol{\delta})$ holds.  We conclude
  that the inclusion $V(\boldsymbol{f})\supset \mathrm{supp}(F)$ implies
  \begin{equation}\label{eq:condition-eqs-incidence-variety}
    (f_0\klk f_r)^d\subset(c_{\boldsymbol{\gamma},\boldsymbol{\delta}}:
    \boldsymbol{\gamma},\boldsymbol{\delta}).
  \end{equation}
  On the other hand, the converse assertion is easily established by
  observing that (\ref{eq:condition-eqs-incidence-variety}) implies
  the corresponding inclusion of radical ideals, which in turn implies
  $V(\boldsymbol{f})\supset \mathrm{supp}(F)$. As a consequence, for elements
  $[F]\in \mathcal{I}_{\ge 1}((\cfq[\G])_d)$ we see that
  (\ref{eq:condition-eqs-incidence-variety}) is equivalent to the
  inclusion $V(\boldsymbol{f})\supset \mathrm{supp}(F)$.

  The inclusion (\ref{eq:condition-eqs-incidence-variety}) is
  equivalent to the membership of all products $f_0^{i_0}\cdots
  f_r^{i_r}$ with $i_0\plp i_r=d$ in the ideal
  $(c_{\boldsymbol{\gamma},\boldsymbol{\delta}}:\boldsymbol{\gamma},
  \boldsymbol{\delta})$.  Fix $(i_0\klk i_r)\in\Z_{\ge
    0}^{r+1}$ with $i_0\plp i_r=d$.  Then $f_0^{i_0}\cdots
  f_r^{i_r}\in(c_{\boldsymbol{\gamma},\boldsymbol{\delta}}:\boldsymbol{\gamma},
  \boldsymbol{\delta})$ if and only if there
  exist homogeneous polynomials $h_{\boldsymbol{\gamma},\boldsymbol{\delta}}
  \in\cfq[X_0\klk X_r]$ of degree $d^2-2d$ for all $|\boldsymbol{\gamma}|=
  |\boldsymbol{\delta}|=d$ with
  \begin{equation}\label{eq:ideal-membership-incidence-var}
    f_0^{i_0}\cdots f_r^{i_r}=\sum_{\boldsymbol{\gamma},\boldsymbol{\delta}}
    h_{\boldsymbol{\gamma},\boldsymbol{\delta}}\,c_{\boldsymbol{\gamma},\boldsymbol{\delta}}.
  \end{equation}
  Equating the corresponding coefficients at both sides of
  (\ref{eq:ideal-membership-incidence-var}) we can reexpress
  (\ref{eq:ideal-membership-incidence-var}) as a linear system with
  the coefficients of the polynomials $h_{\boldsymbol{\gamma},\boldsymbol{\delta}}$ as
  indeterminates.  The number of equations of this system equals
  the number of coefficients of the polynomials on both sides of
  (\ref{eq:ideal-membership-incidence-var}). These are homogeneous
  polynomials of degree $d^2$ in $r$ indeterminates, having thus at
  most $\binom{d^2+r}{r}$ nonzero coefficients. Then we have at most
  $\binom{d^2+r}{r}$ equations. On the other hand, the number of
  unknowns is equal to the number of coefficients of all the
  polynomials $h_{\boldsymbol{\gamma},\boldsymbol{\delta}}$, namely
  $\binom{d^2-2d+r}{r}\binom{d+r}{r}{}^2$. We also remark that the coefficients of the
  matrix of this system are linear combinations of the coefficients
  $\boldsymbol{a}\defineAs (a_{\boldsymbol{\alpha},
  \boldsymbol{\beta}})_{\boldsymbol{\alpha},\boldsymbol{\beta}}$ of $F$.

  The existence of solutions of
  (\ref{eq:ideal-membership-incidence-var}) is equivalent to the
  identity of the ranks of the coefficient matrix and the extended
  coefficient matrix of
  (\ref{eq:ideal-membership-incidence-var}). Since these two matrices
  have rank at most $\binom{d^2+r}{r}$, the existence of solutions of
  (\ref{eq:ideal-membership-incidence-var}) is equivalent to the
  vanishing of the determinant of certain minors of size at most
  $\big(\binom{d^2+r}{r}+1\times\binom{d^2+r}{r}+1\big)$. The entries
  of such minors consist of linear combinations of the coefficients of
  the product $f_0^{i_0}\cdots f_r^{i_r}$ in their last column and
  linear combinations of the coefficients of the polynomials
  $c_{\boldsymbol{\gamma},\boldsymbol{\delta}}$ in the remaining columns. It follows that their
  determinants are bihomogeneous polynomials of degree
  $\binom{d^2+r}{r}$ in the vector $\boldsymbol{a}\defineAs (a_{\boldsymbol{\alpha},
    \boldsymbol{\beta}})_{\boldsymbol{\alpha},\boldsymbol{\beta}}$
  of coefficients of $F$ and degree $d$ in the coefficients of the
  polynomials $f_0\klk f_r$. This finishes the proof of the
  proposition.
\qed\end{proof}

Now we are in position to obtain an upper bound on the degree of the
incidence variety $\Gamma$ from
(\ref{eq:definition-incidence-variety}).
\begin{theorem}\label{th:degree-incidence-variety}
  The following upper bound holds for $d\ge r\ge 3$:
  $$\deg\Gamma\le (ed)^{r(r+1)(d^2+1)+3rg_{d,r}},$$
  where $e$ denotes the basis of the natural logarithm and
  \begin{equation}\label{eq:dim-homogeneous-piece-grassmannian}
    g_{d,r}\defineAs \binom{r+d-2}{d}^2 \cdot \frac{r+d-1}{(r-1)(d+1)}.
  \end{equation}

\end{theorem}
\begin{proof}
  By our previous remarks it turns out that the degree of the
  incidence variety $\Gamma$ equals the degree of the open dense
  subset $\theta^{-1}(\mathcal{V})$ of $\Gamma$.

  Let $\G\defineAs \G(r-2,2)$ and let $\mathcal{I}((\cfq[\G])_d)$ be
  the open subset of $(\cfq[\G])_d$ corresponding to irreducible
  cycles.  According to \cite[Exercise I.3.28.6]{Kollar99}, the set
  $\mathcal{I}_{\ge 1}((\cfq[\G])_d)$ of cycles of
  $\mathcal{I}((\cfq[\G])_d)$ of dimension 1 is a closed subset of
  $\mathcal{I}((\cfq[\G])_d)$ which is described by equations in the
  coefficients $(a_{\boldsymbol{\alpha},\boldsymbol{\beta}})_{\boldsymbol{\alpha},\boldsymbol{\beta}}$ of the cycles of
  degree $\binom{2dr+d}{r}$. From Lemma
  \ref{lemma:desigualdad-bezout-andreotti} it follows that the union
  $W_c$ of the irreducible components of $\mathcal{I}_{\ge
    1}((\cfq[\G])_d)$ of codimension at most $c$ in
  $\mathcal{I}((\cfq[\G])_d)$ has degree bounded by $\binom{2dr+d}{r}{}^c$.

  Fix an integer $c\ge 0$ and consider the restriction
  $\theta|_{\Gamma_c}$ of $\theta$ to the (nonempty) incidence variety
  $\Gamma_c\defineAs \Gamma\cap(\mathcal{V}\times
  W_c)=\theta^{-1}(\mathcal{V})\cap(\mathcal{V}\times W_c)$. To a
  generic cycle $[F]$ of an irreducible component of $W_c$ corresponds
  a unique $\boldsymbol{f}\in\mathcal{V}$ such that $(\boldsymbol{f},[F])\in\Gamma_c$
  holds. This shows that $\dim\Gamma_c=\dim W_c$.

  Denote $g_{d,r}\defineAs \dim (\cfq[\G])_d$ and observe that the
  following identity holds:
  $$\dim (\cfq[\G])_d =\binom{r+d-2}{d}^2-
  \binom{r+d-2}{d-1}\binom{r+d-2}{d+1} \defineAs g_{d,r}$$
  (see, e.g., \cite{GhKr04}). Then we have the following upper bound
  on $\mathrm{codim}_{\mathcal{V}\times W_c}\Gamma_c$:
  $$
  \mathrm{codim}_{\mathcal{V}\times W_c}\Gamma_c=\dim\mathcal{V}+\dim
  W_c-\dim\Gamma_c=\dim\mathcal{V}\le g_{d,r}.
  $$
  %
  Proposition \ref{prop:degree-eqs-incidence-var} says that
  $\theta^{-1}(\mathcal{V})$ is defined in the product
  $\mathcal{V}\times \mathcal{I}_{\ge 1}(\cfq[\G])_d$ by equations of
  bidegree at most $(d,D_r)$, where $D_r\defineAs\binom{d^2+r}{r}$.
  Therefore, from Lemma \ref{lemma:desigualdad-bezout-andreotti} we
  conclude that
  $$\deg\Gamma_c\le
  \deg(\mathcal{V}\times W_c)(d+D_r)^{g_{d,r}}\le
  \deg\mathcal{V}\,\deg W_c\,(d+D_r)^{g_{d,r}}.$$
  By Lemma \ref{lemma:estimate-degree-gamma-theta} we have
  $\deg\mathcal{V}=\deg\theta(\Gamma)\le\binom{rd+d}{r}{}^{(r+1)(d^2+1)}$. As a consequence,
  \begin{equation}\label{eq:aux-proof-upper-bound-deg-incidence}
    \deg\theta^{-1}(\mathcal{V})=
    \deg \Gamma_{g_{d,r}}
    \le \binom{rd+d}{r}^{(r+1)(d^2+1)}\binom{2dr+d}{r}^{g_{d,r}}(d+D_r)^{g_{d,r}}.
  \end{equation}

  First we observe that the inequality $d+D_r\le d+(e(d^2+r)/r)^r\le
  d^{2r}$ holds for $d\ge r\ge 3$. On the other hand, from
  \cite[Theorem 2.6]{Stanica01} we easily deduce the following upper
  bounds:
  $$\binom{rd+d}{r}\le (ed)^r, \quad \binom{2dr+d}{r}\le (2ed)^r.$$
  Combining these inequalities with
  (\ref{eq:aux-proof-upper-bound-deg-incidence}) and Remark
  \ref{rem:expression-d-coeff-hilb-pol-grassmannian} below, the bound
  of the statement of the theorem follows.
\qed\end{proof}

\begin{remark}\label{rem:expression-d-coeff-hilb-pol-grassmannian}
  The following identities hold:
  \begin{align*}
    g_{d,r}& = \binom{r+d-2}{d}^2- \binom{r+d-2}{d-1}\binom{r+d-2}{d+1}\\
    & = \frac{1}{d+1}\binom{d+r-2}{r-2}\binom{d+r-1}{r-1}=\prod_{i=1}^{r-2}\frac{d+r-i-1}{r-i-1}\frac{d+r-i}{r-i}.
  \end{align*}
\end{remark}

As a consequence of this result we obtain an upper bound on the degree
of the union $\widetilde{\mathcal{C}}_{d,r}$ of the absolutely
irreducible components of the Chow variety $\mathcal{C}_{d,r}$
containing an absolutely irreducible nondegenerate curve. Such a bound
is deduced from the facts that $\widetilde{\mathcal{C}}_{d,r}$ is the
image of the linear projection
$\phi:\Gamma\to\widetilde{\mathcal{C}}_{d,r}$ and that the degree does
not increase under linear mappings.
\begin{corollary}\label{coro:degree-restricted-chow}
  For $d\ge r\ge 3$, the following upper bound holds:
  $$\deg \widetilde{\mathcal{C}}_{d,r}\le
  (ed)^{r(r+1)(d^2+1)+3rg_{d,r}},$$
  with $g_{d,r}$ as in (\ref{eq:dim-homogeneous-piece-grassmannian}).
\end{corollary}
%
%
\subsubsection{From an upper bound on the degree of
  $\widetilde{\mathcal{C}}_{d,r}$ to one for $\mathcal{C}_{d,r}$}
Next we obtain an upper bound on the degree of the union
$\widehat{\mathcal{C}}_{d,r}$ of the components of the Chow variety
$\mathcal{C}_{d,r}$ for which the generic point corresponds to an
absolutely irreducible curve. Since we have an upper bound on the
degree of the restricted Chow variety
$\widetilde{\mathcal{C}}_{d,r}$, namely the union of the components
of $\mathcal{C}_{d,r}$ for which the generic point corresponds to a
nondegenerate absolutely irreducible curve, there only remains to
consider the degenerate cases.

Fix $k$ with $2\le k<r$ and consider the union
$\widehat{\mathcal{C}}_{d,r}^{(k)}$ of the absolutely irreducible
components of $\widehat{\mathcal{C}}_{d,r}$ for which the generic
point corresponds to an absolutely irreducible curve spanning a
$k$--dimensional linear subspace of $\Pp^r$. 
%
%
\begin{lemma}\label{lemma:degree-degenerate-cases}
  For $d\ge k\ge 2$ and $r\ge 3$, the following upper bound holds:
  \begin{equation}
    \deg\,  \widehat{\mathcal{C}}_{d,r}^{(k)} \le (d^2+1)^{ \dim
      \widehat{\mathcal{C}}_{d,r}^{(k)}}(k+1)(r-k)\deg
    \,\widetilde{\mathcal{C}}_{d,k}.
  \end{equation}
\end{lemma}

\begin{proof}
  First we observe that it is easy to construct a set--theoretic
  bijection
  $$\widetilde{\mathcal{C}}_{d,k} \times \G(k,r)\longleftrightarrow
  \widehat{\mathcal{C}}_{d,r}^{(k)}.
  $$
  Indeed, a curve in $\Pp^k$ can be embedded in $\Pp^r$ and moved to
  any subspace of dimension $k$ using a suitable linear isomorphism.
  Fixing the embedding, a bijection as above is obtained. We claim
  that there exists a dense open subset $U$ of $\G(k,r)$ such that the
  restriction of a set--theoretic bijection as above to
  $\widetilde{\mathcal{C}}_{d,k}\times U$ is given by a polynomial map
  $\phi:\Pp\V_{d,k}\times U \to\Pp\V_{d,r}$ such that
  $\phi(\widetilde{\mathcal{C}}_{d,k}\times U)$ contains a dense open
  subset of $\widehat{\mathcal{C}}_{d,r}^{(k)}$.

  Fix a basis $\{e_0\klk e_r\}$ of $\cfq^{r+1}$, let $\Pp^k$ be
  embedded in $\Pp^r$ as the subspace spanned by the first $k+1$ basis
  vectors $e_0,\dots,e_k$ in $\Pp^r$ and consider the corresponding
  embedding of $\Pp\V_{d,k}$ in $\Pp\V_{d,r}$. If $[F]\in\Pp\V_{d,r}$
  is the Chow form of a cycle $C\in\mathcal{C}_{d,r}$, then $C$ is
  contained in $\Pp^k$ if and only if $[F]$ depends on $A_0,\dots,A_k$
  and $B_0,\dots,B_k$ and not on $A_{k+1},\dots,A_r$ and
  $B_{k+1},\dots,B_r$. Let $\Phi$ be the linear space of $\Pp^r$
  spanned by the last $r-k$ basis vectors $e_{k+1}\klk e_r$ and let
  $U_\Phi$ be the affine open subset of $\G(k,r)$ consisting of the
  subspaces complementary to $\Phi$. Then every $\Lambda\in U_\Phi$ is
  represented as the row space of a unique matrix of the form
  \begin{equation}
    \left(\begin{array}{cccccccc}
        1 & 0 & \dots & 0 & m_{0,1} & m_{0,2}& \dots & m_{0,r-k}\\
        0 & 1 & \dots & 0 & m_{1,1} & m_{1,2}& \dots & m_{1,r-k}\\
        \vdots&   &  &  &   &  & \dots &  \\
        0& 0 & \dots & 1 & m_{k,1} & m_{k,2}& \dots & m_{k,r-k}
      \end{array}\right)
  \end{equation}
  and viceversa. The entries $m_{i,j}$ of the last $r-k$ columns of
  this matrix yield a bijection of $U_\Phi$ with $\A^{(k+1)(r-k)}$,
  and are known as the Pl\"{u}cker coordinates of the Grassmannian
  $\G(k,r)$ (see, e.g., \cite{Harris92}).

  For each $[F]\in\Pp\V_{d,k}$ and $(m_{i,j})\in \A^{(k+1)(r-k)}$, we
  define
  $$\phi([F],(m_{i,j}))(\boldsymbol{A},\boldsymbol{B})\defineAs
  [F(\mathcal{M}^{-1}\boldsymbol{A},\mathcal{M}^{-1}\boldsymbol{B})],$$
  where $\mathcal{M}\in\A^{(r+1)\times(r+1)}$ is the matrix
  \begin{equation}\label{eq:matrix plucker coordinates}
    \mathcal{M}\defineAs \left(\begin{array}{cc}{\rm Id}_{k+1}  & (m_{i,j}) \\
        \mathbf{0} & {\rm Id}_{r-k}
      \end{array}\right),\end{equation}
  $\mathrm{Id}_j$ denotes the identity matrix of $\A^{j\times j}$
  for every $j\in\N$ and $\mathbf{0}$ denotes the zero matrix of
  $\A^{(r-k)\times(k-1)}$. Since the identity
  $$\mathcal{M}^{-1}=\left(\begin{array}{cc}
      \mathrm{Id}_{k+1} & -(m_{i,j}) \\
      \mathbf{0} & {\rm Id}_{r-k}
    \end{array}\right)$$
  holds, we easily conclude that the injection $\phi$ is a regular map
  defined by polynomials of degree at most $d^2+1$.

  If $[F]$ is the Chow form of an irreducible nondegenerate curve $C$
  of $\Pp^k$, then we have that $\phi([F],(m_{i,j}))$ is the Chow form of the curve
  $C_\mathcal{M}\defineAs \{\mathcal{M}\boldsymbol{x}:\boldsymbol{x}\in C\}$.
  Clearly, $C_\mathcal{M}$ is an irreducible
  curve which is nondegenerate in the subspace spanned by the first
  $k+1$ rows of $\mathcal{M}$. This shows that
  $\phi(\widetilde{\mathcal{C}}_{d,k}\times
  \A^{(k+1)(r-k)})\subset\widehat{\mathcal{C}}_{d,r}^{(k)}$ holds.

  Now, let $V_\Phi$ be the open dense subset of
  $\widehat{\mathcal{C}}_{d,r}^{(k)}$ consisting of the forms whose
  support spans a subspace complementary to $\Phi$. For $[G]$ in
  $V_\Phi$, consider the Pl\"ucker coordinates
  $(m_{i,j})\in\A^{(k+1)(r-k)}$ of the subspace spanned by
  $\mathrm{supp}(G)$ and the corresponding matrix $\mathcal{M}$ defined as in
  (\ref{eq:matrix plucker coordinates}). By reversing the argument
  above, it turns out that the polynomial $F(\boldsymbol{A},\boldsymbol{B})
  \defineAs G(\mathcal{M}\boldsymbol{A},\mathcal{M}\boldsymbol{B})$
  depends only on the indeterminates $A_0,\dots,A_k$ and $B_0,\dots,
  B_k$, and hence its support is a nondegenerate curve in $\Pp^k$. We
  conclude that $[F]$ belongs to $\widetilde{\mathcal{C}}_{d,k}$ and
  the Chow form $[G]$ is the image under $\phi$ of the pair
  $([F],(m_{i,j}))$.  It follows that
  $\phi(\widetilde{\mathcal{C}}_{d,k}\times \A^{(k+1)(r-k)})$ contains
  a dense open subset of $\widehat{\mathcal{C}}_{d,r}^{(k)}$, as
  claimed.

  From our claim we deduce that
  $$\deg\phi(\widetilde{\mathcal{C}}_{d,k}\times
  \A^{(k+1)(r-k)})=\deg \widehat{\mathcal{C}}_{d,r}^{(k)}.$$
  Applying Lemma \ref{lemma:bound-deg-morphism}, the estimate of the
  lemma follows.
\qed\end{proof}
\begin{proposition}\label{prop:degree-chow-var-irred-curves}
  For $d\ge 1$ and $r\ge 3$, the following upper bound holds:
  $$\deg\widehat{\mathcal{C}}_{d,r}
  \le 2(ed)^{r(r+1)(d^2+1)+3rg_{d,r}},$$
  where $g_{d,r}$ is defined as in
  (\ref{eq:dim-homogeneous-piece-grassmannian}).
\end{proposition}
\begin{proof}
  First suppose that $d\ge r$. From Fact
  \ref{fact:dimension-chow-var}, Corollary
  \ref{coro:degree-restricted-chow} and Lemma
  \ref{lemma:degree-degenerate-cases} we have
  $$
  \sum_{k=2}^{r-1}\deg\widehat{\mathcal{C}}_{d,r}^{(k)} \le
  (d^2+1)^{2d(r-1) + d(d+3)/2}\sum_{k=2}^{r-1}
  (k+1)(r-k)(ed)^{k(k+1)(d^2+1)+3kg_{d,k}}.
  $$
  By Remark \ref{rem:expression-d-coeff-hilb-pol-grassmannian} we
  easily deduce that the numbers $g_{d,k}$ are increasing functions of
  $k$, which implies
  $$(ed)^{(k+1)k(d^2+1)+3kg_{d,k}}\le
  (ed)^{r(r-1)(d^2+1)+3rg_{d,r}}$$
  for $2\le k\le r-1$. This shows that
  \begin{align*}
    \begin{array}{rcl}
      \displaystyle \sum_{k=2}^{r-1}\deg\widehat{\mathcal{C}}_{d,r}^{(k)}
      &\le&
      (r-2)r^2(d^2+1)^{2d(r-1) + d(d+3)/2}(ed)^{r(r-1)(d^2+1)+3rg_{d,r}}\\
      &\le& (ed)^{r(r+1)(d^2+1)+3rg_{d,r}}.
    \end{array}
  \end{align*}
  Since $\deg\widehat{\mathcal{C}}_{d,r}\le
  \sum_{k=2}^{r-1}\deg\widehat{\mathcal{C}}_{d,r}^{(k)}+
  \deg\widetilde{\mathcal{C}}_{d,r}$, from the previous bound and
  Corollary \ref{coro:degree-restricted-chow} we deduce the statement
  of the proposition for $d\ge r$.

  Next suppose that $2\le d<r$. Since an irreducible nondegenerate
  curve in $\Pp^k$ has degree at least $k$ (see, e.g.,
  \cite[Proposition 18.9]{Harris92}), we conclude that
  $\widehat{\mathcal{C}}_{d,r}^{(k)}$ is empty for $k>d$ and
  $\widetilde{\mathcal{C}}_{d,r}$ is also empty. This implies
  $$\deg\widehat{\mathcal{C}}_{d,r}=\sum_{k=2}^{d}
    \deg\widehat{\mathcal{C}}_{d,r}^{(k)}\le
    \sum_{k=2}^{d} (ed)^{(k+1)r(d^2+1)+3kg_{d,k}}\le
    (d-2)(ed)^{(d+1)r(d^2+1)+3rg_{d,r}}$$
  and proves the proposition in this case.

  Finally, if $d=1$, then $\widehat{\mathcal{C}}_{1,r}=\G(1,r)$.  In
  this case we have an explicit expression for the degree of
  $\G(1,r)$, from which the estimate of the statement follows (see,
  e.g., \cite[Example 19.14]{Harris92}):
  $$\deg\widehat{\mathcal{C}}_{1,r}=
  \frac{(2r-2)!}{(r-1)!\,r!}=\frac{1}{r}\binom{2r-2}{r-1} \le
  (2e)^{r-1}\le 2e^{2r(r+1)+3rg_{1,r}}.$$
  This finishes the proof of the proposition.
\qed\end{proof}
Finally, we obtain an upper bound on the degree of the Chow variety
$\mathcal{C}_{d,r}$ of curves of $\Pp^r$ of degree $d$. We recall the
quantity $g_{d,r}$ from (\ref{eq:dim-homogeneous-piece-grassmannian}).
\begin{theorem}\label{th:upper-bound-deg-chow-var}
  For $d\ge 1$ and $r\ge 3$, we set
  \begin{equation}\label{eq:cdr}
    c_{d,r}= (2ed)^{r(r+1)(d^2+1)+4rg_{d,r}}.
  \end{equation}
  Then $\deg\mathcal{C}_{d,r} \leq c_{d,r}.$
\end{theorem}
\begin{proof}
  Let $(\boldsymbol{a},\boldsymbol{d})\defineAs (a_1\klk a_s,d_1\klk d_s)$ be a vector of
  positive integers with $d_1\ge d_2\ge\cdots\ge d_s$ and $a_1d_1\plp
  a_sd_s=d$, and consider the morphism
  $$\begin{array}{rcl}
    \mu_{(\boldsymbol{a},\boldsymbol{d})}:\ \widehat{\mathcal{C}}_{d_1,r}\times\cdots
    \times\widehat{\mathcal{C}}_{d_s,r}\ &\to&\mathcal{C}_{d,r}\\
    ([F_1],\dots,[F_s])&\mapsto&[\prod_{i=1}^sF_i^{a_i}].
  \end{array}$$
  For $(\boldsymbol{a},\boldsymbol{d})\defineAs (a_1\klk a_s,d_1\klk d_s)$ as before, the
  numbers $s$ and $d_1\plp d_s$ are called the length and the weight
  of $\boldsymbol{d}$ and are denoted by $\ell(\boldsymbol{d})$ and $w(\boldsymbol{d})$, respectively. If
  $\mathcal{D}$ denotes the set of all $(\boldsymbol{a},\boldsymbol{d})$ with $d_1\ge
  d_2\ge\cdots\ge d_{\ell(\boldsymbol{d})}$ and $a_1d_1\plp a_sd_s=d$, then it is
  clear that
  $$\mathcal{C}_{d,r}=\widehat{\mathcal{C}}_{d,r}\cup
  \mathcal{R}_{d,r}=
  \bigcup_{(\boldsymbol{a},\boldsymbol{d})\in\mathcal{D}}\mathrm{im}\,\mu_{(\boldsymbol{a},\boldsymbol{d})}.$$
  Furthermore, since each image $\mathrm{im}\,\mu_{(\boldsymbol{a},\boldsymbol{d})}$ is a closed
  subset of $\mathcal{C}_{d,r}$, we have
  \begin{equation}\label{eq:aux-proof-upper-bound-deg-chow-var}
    \deg\mathcal{C}_{d,r}\le \sum_{(\boldsymbol{a},\boldsymbol{d})\in\mathcal{D}}\deg
    \mathrm{im}\,\mu_{(\boldsymbol{a},\boldsymbol{d})}.
  \end{equation}

  Applying Lemma \ref{lemma:bound-deg-morphism} and Proposition
  \ref{prop:degree-chow-var-irred-curves}, we obtain the following
  inequality:
  $$\deg \mathrm{im}\,\mu_{(\boldsymbol{a},\boldsymbol{d})}\le
  d^{g_{d,r}}\prod_{1 \le i\le \ell(\boldsymbol{d})}\deg
  \widehat{\mathcal{C}}_{d_i,r}\le d^{g_{d,r}}\prod_{1 \le i \le
    \ell(\boldsymbol{d})}2(ed_i)^{r(r+1)(d_i^2+1)+3rg_{d_i,r}}.$$
  Let $c_{\boldsymbol{d}}\defineAs
  \prod_{i=1}^{\ell(\boldsymbol{d})}2(ed_i)^{r(r+1)(d_i^2+1)+3rg_{d_i,r}}$ for any
  $\boldsymbol{d}\defineAs (d_1\klk d_s)$ with $w(\boldsymbol{d})\le d$.
  %
  \begin{claim}
    $\sum_{(\boldsymbol{a},\boldsymbol{d})\in\mathcal{D}} c_{\boldsymbol{d}} \le
    (2ed)^{r(r+1)(d^2+1)+3rg_{d,r}}$.
  \end{claim}
  \begin{poc}
    %
    Define $\widehat{c}_k\defineAs \exp(h(k))$, where $h\colon \R_{\geq 0}\to\R$ is the function
    defined by the identity $\exp(h(x))\defineAs
    2(ex)^{r(r+1)(x^2+1)+3rg_{x,r}}$.  From Remark
    \ref{rem:expression-d-coeff-hilb-pol-grassmannian} we easily conclude
    that $h$ is differentiable and its derivative is increasing. A
    straightforward argument shows the inequality
    $\widehat{c}_k\widehat{c}_{m}\le\widehat{c}_{m+k}$ for arbitrary
    positive integers $k,m$ and $r\ge 3$. Hence, for $\boldsymbol{d}\defineAs (d_1\klk
    d_s)$ with $s\ge 2$,
    $$c_{\boldsymbol{d}}\defineAs \widehat{c}_{d_1}\cdots \widehat{c}_{d_s}\le
    \widehat{c}_{w(\boldsymbol{d})}.$$
    This shows that
    $$\sum_{\{(\boldsymbol{a},\boldsymbol{d}):w(\boldsymbol{d})=m\}}
    c_{(\boldsymbol{a},\boldsymbol{d})}
    \le\#\{(\boldsymbol{a},\boldsymbol{d}):w(\boldsymbol{d})=m\}\cdot \widehat{c}_{m}
    \le 2^{m+d}\,\widehat{c}_m.
    $$
    Taking into account that the expression $2^m\widehat{c}_m$ is
    increasing in $m$, we obtain
    $$
      \sum_{(\boldsymbol{a},\boldsymbol{d})\in\mathcal{D}}\!c_{\boldsymbol{d}}
      =\sum_{m=1}^d\sum_{\{(\boldsymbol{a},\boldsymbol{d}):\,w(\boldsymbol{d})=m\}}
      c_{(\boldsymbol{a},\boldsymbol{d})}\le \sum_{m=1}^d2^{m+d}\widehat{c}_m
       \le d\,2^{2d}\widehat{c}_d  \le (2ed)^{r(r+1)(d^2+1)+3rg_{d,r}}.
    $$
    This proves the claim.
  \end{poc}
  Combining (\ref{eq:aux-proof-upper-bound-deg-chow-var}) with this
  claim proves the theorem.
\qed\end{proof}
%
%
\section{The number of $\fq$-reducible curves}\label{sec:RedCyc}
From Theorem \ref{th:upper-bound-deg-chow-var} we derive an upper
bound on the number of $\fq$-reducible cycles of the Chow variety
$\mathcal{C}_{d,r}$.
\begin{theorem}\label{th:upper-bound-red-curves}
  For $r\ge 3$ and $d\ge \threshold $, the following upper bound
  holds:
  \begin{align*}
    \#\mathcal{R}_{d,r}^{(q)}(\fq)\le
    \begin{cases}
      c_{d,r}q^{b_{d,r}-r+2} & \text{if} \ d=\threshold ,\\
      c_{d,r}q^{b_{d,r}-(d-2r+3)} & \text{otherwise},
    \end{cases}
  \end{align*}
  with $b_{d,r}\defineAs \dim\mathcal{C}_{d,r}$ and $c_{d,r}$ as in
  Fact \ref{fact:dimension-chow-var} and (\ref{eq:cdr}), respectively.
\end{theorem}
\begin{proof}
  Let $c_{k,r}\defineAs (2ek)^{r(r+1)(k^2+1)+4rg_{k,r}}$ for $k\in\N$.
  According to (\ref{eq:expression-upper-bound-red-curves}) and
  Theorem \ref{th:upper-bound-deg-chow-var}, we have the inequality
  \begin{equation}\label{eq:aux-proof-upper-bound-red-curves}
    \#\mathcal{R}_{d,r}^{(q)}(\fq)\le\left\{\begin{array}{ll}
        4 \sum_{1 \le k \le d/2} ~ c_{k,r}\,c_{d-k,r}
        q^{b_{d,r}-r+2}& \text{if} \ d=4r-8,\\
        4 \sum_{1 \le k\le d/2}  ~ c_{k,r}\,c_{d-k,r}q^{b_{d,r}-d+2r-3}&
        \text{otherwise}.
      \end{array}\right.
  \end{equation}
  %
  Let $h \colon \R_{\geq 0}\to\R$ be the function defined by the
  identity $\exp(h(x))=(2ex)^{r(r+1)(x^2+1)+4g_{x,r}}$. Arguing as in
  the proof of Theorem \ref{th:upper-bound-deg-chow-var}, we deduce
  that $h$ is differentiable and its derivative is increasing. This
  implies that the function $f_y:[0,y/2]\to\mathbb{R}_{\ge 0}$ defined
  by $f_y(x)\defineAs \exp(h(x))\exp(h(y-x))$ is decreasing for any
  positive real number $y$. It follows that $c_{k,r}\,c_{d-k,r} \le
  c_{1,r}\,c_{d-1,r}$ for $2 \le k \le d/2 $, and hence
  \begin{equation}\label{eq:aux-proof-upper-bound-red-curves2}
    \sum_{1 \le k \le d/2} c_{k,r}\,c_{d-k,r}\le \frac{d}{2}\,c_{1,r}\,c_{d-1,r}
    \le \frac{c_{d,r}}{4} .
  \end{equation}
  Combining (\ref{eq:aux-proof-upper-bound-red-curves}) and
  (\ref{eq:aux-proof-upper-bound-red-curves2}) we easily deduce the
  statement of the theorem.
\qed\end{proof}
In order to obtain bounds on the probability that an $\fq$-curve in
$\Pp^r$ is $\fq$-reducible, we take as a lower bound on the number
of all $\fq$-curves of $\Pp^r$ the number $\#P(d,r)(\fq)$ of plane
$\fq$-curves in $\Pp^r$. Bounds on the number $\#R(d,r)(\fq)$ of
plane $\fq$-reducible curves of $\Pp^r$ are provided by the
estimates for irreducible bivariate and multivariate polynomials of
\cite{Gathen08} and \cite{GaViZi10}.
For the homogeneous case, these estimates imply that the number $\#
R(d,2)$ of $\fq$-reducible curves in $P(d,2)$ is bounded as
\begin{equation}\label{eq:lower_bound_irred_proj_plane_curves}
  \#P(d,2)\cdot(q-3)q^{-d}\le \# R(d,2)\le\#P(d,2)\cdot (q+2)q^{-d}.
\end{equation}

\def\bdr{a}
\begin{lemma}\label{lemma:lower-bound-number-plane-curves}
  Let $r \ge 3$ and $d \geq 4r-8$. Then
  \begin{align*}
    q^{b_{d,r}}=q^{3(r-2)+d(d+3)/2} & \le  \#P(d,r)(\mathbb{F}_{q})\le7q^{b_{d,r}},\\
     \#R(d,r)(\mathbb{F}_{q}) & \le 13q^{b_{d,r}-d+1},\\
    q^{b_{d,r}}  \leq \#\mathcal{C}_{d,r}(\F_{q}) & <
    2c_{d,r}q^{b_{d,r}}.
  \end{align*}
\end{lemma}

\begin{proof}
  We fix $\mathbb{F}_{q}$, drop it from the notation, and consider the
  incidence variety
  \begin{equation}\label{eq:incidence_variety_bounds_plane_curves}
    I = \{(C,E)\in P(d,r) \times \mathbb{G}(2,r) \colon C \subseteq E
    \}.
  \end{equation}

  The second projection $\pi_{2} \colon I \longrightarrow
  \mathbb{G}(2,r)$ is surjective, and all its fibers are isomorphic to
  the variety $P(d,2)$ of plane curves of degree $d$. The latter are
  parametrized by the nonzero homogeneous trivariate polynomials of
  degree $d$, and
  $$
  \#P(d,2) = \frac{q^{(d+2)(d+1)/2}-1}{q-1},
  $$
  $$
  \#I = \#P(d,2)\cdot \#G(2,r)= \#P(d,2) \cdot
  \frac{(q^{r+1}-1)(q^{r+1}-q)(q^{r+1}-q^2)}{
    (q^3-1)(q^3-q)(q^3-q^2)}.
  $$
  The fibers of the first projection $\pi_{1} \colon I \longrightarrow
  P(d,r)$ usually are singletons. The only exceptions occur when $C= d
  \cdot L$ equals $d$ times a line $L$.  There are
  $$
  \#\mathbb{G}(1,r)= \frac{(q^{r+1}-1)(q^{r+1}-q)}{(q^{2}-1)(q^{2}-q)}
  $$
  such $d \cdot L$, and their fiber size is
  $$
  \#\pi_{1}^{-1}(d \cdot L)= \frac{q^{r+1}-q^{2}}{q^{3}-q^{2}}.
  $$
  It follows that
  \begin{align}\label{al:PG}
    \#P(d,r)&= \#I - \#\mathbb{G}(1,r) \cdot
    \bigg(\frac{q^{r+1}-q^{2}}{q^{3}-q^{2}}-1\bigg)\nonumber\\
    & = \#\mathbb{G}(1,r)\cdot \bigg(\frac{\#P(d,2)\cdot
      (q^{r+1}-q^2)(q^2-1)(q^2-q)}{(q^3-1)(q^3-q)(q^{3}-q^{2})}-\frac{q^{r+1}
      -q^{3}}{q^{3}-q^{2}}\bigg)\nonumber\\
    & = \# \mathbb{G}(1,r) \cdot
    \frac{(q^{(d+2)(d+1)/2}-1)(q^{r-1}-1)-(q^{r-1}-q)(q^{3}-1)}{(q^{3}-1)(q-1)}
    \nonumber\\
    & = q^{3(r-2)+d(d+3)/2}\frac{(1-q^{-r})(1-q^{-r-1})}{(1-q^{-1})^{2}(1-q^{-2})(1-q^{-3})}\nonumber\\
    & \quad \cdot\big((1-q^{-c})(1-q^{-r+1})-q^{3-c}(1-q^{-r+2})(1-q^{-3})\big),
  \end{align}
  where $c=(d+2)(d+1)/2$. Since $q$, $d\ge 2$, we have $c \ge 6$ and
  hence
  \begin{align*}
    2q & \ge 1+q^{5-c}+q^{2-c}+q^{3-r},\\
    (1-q^{-c})(1-q^{-r+1})-q^{3-c} & \ge (1-q^{-1})^{2}.
  \end{align*}
  Therefore, the last factor in (\ref{al:PG}) is at least
  $(1-q^{-1})^{2}$, which implies
  $$\#P(d,r)  \ge q^{3(r-2)+d(d+3)/2}.$$
  On the other hand, from (\ref{al:PG}) we also deduce that
  \begin{align*}
    \#P(d,r) \le&\
    q^{b_{d,r}}\frac{(1-q^{-c})(1-q^{-r-1})(1-q^{-r})(1-q^{-r+1})}{
      (1-q^{-1})^{2}(1-q^{-2})(1-q^{-3})}\\
    \le&\ q^{b_{d,r}}\frac{1}{ (1-q^{-1})^{2}(1-q^{-2})(1-q^{-3})}\le
    7 q^{b_{d,r}}.
  \end{align*}
  This proves the bounds for $P(d,r)$.
  Concerning $R(d,r)$, we consider, instead of the incidence variety of
  (\ref{eq:incidence_variety_bounds_plane_curves}), the ``restricted''
  incidence variety
  $$
  I_R= \{(C,E)\in R(d,r) \times \mathbb{G}(2,r) \colon C \subseteq E
  \}.
  $$
  Arguing as before and applying
  (\ref{eq:lower_bound_irred_proj_plane_curves}), we obtain
  \begin{align*}
    \#R(d,r)& \leq   \#I =  \#R(d,2)\cdot \#G(2,r)\\
    &\le \#P(d,2) \cdot\frac{q+2}{q^d}\cdot
    \frac{(q^{r+1}-1)(q^{r+1}-q)(q^{r+1}-q^2)}{
      (q^3-1)(q^3-q)(q^3-q^2)}\\
    &= q^{b_{d,r}-d+1}
    \frac{(1-q^{-(d+1)(d+2)/2})(1+2q^{-1})(1-q^{-r-1})(1-q^{-r})(1-q^{-r+1})}{
      (1-q^{-3})(1-q^{-2})(1-q^{-1})^2}\\
    & < q^{b_{d,r}-d+1}\frac{1+2q^{-1}}{(1-q^{-3})
      (1-q^{-2})(1-q^{-1})^2}\ \, \le\ 13q^{b_{d,r}-d+1}.
  \end{align*}
  The lower bound for $\mathcal{C}_{d,r}$ follows from the fact that
  $P(d,r) \subseteq \mathcal{C}_{d,r}$, and the upper bound from
  (\ref{eq:upper-bound-number-rat-points}) and Theorem
  \ref{th:upper-bound-deg-chow-var}.
\qed\end{proof}

We find the following bounds on the probability that a random curve in
$\mathcal{C}_{d,r}(\fq)$ is $\fq$-reducible.

%
%
\begin{theorem}\label{coro:upper-bound-prob-red-curves}
  \begin{enumerate}
  \item[$(1)$]\label{th:u-b-p-r-c-1} If $r\ge 3$ and $d\ge 4r-7 $, then
    $$
    \frac{(1-13q^{2-d})}{2c_{d,r}}q^{-(d-2r+3)}\le \frac{
      \#\mathcal{R}_{d,r}^{(q)}(\fq)}{ \#\mathcal{C}_{d,r}(\fq)}\le
    c_{d,r}\, q^{-(d-2r+3)},
    $$
    with $c_{d,r}$ as in (\ref{eq:cdr}). If $d \geq 7$, then also
    $$
    \frac{1}{4c_{d,r}}q^{-(d-2r+3)} \le \frac{
      \#\mathcal{R}_{d,r}^{(q)}(\fq)}{ \#\mathcal{C}_{d,r}(\fq)}\le
    c_{d,r}\, q^{-(d-2r+3)}.
    $$
  \item[$(2)$]\label{th:u-b-p-r-c-2} If $r \geq 3$ and $d = 4r-8$, then
    $$
    \frac{1}{2\,d!\, c_{r,d}}  q^{-r+2} \leq \frac{\#
      \mathcal{R}_{d,r}^{(q)}
      (\mathbb{F}_{q})}{\#\mathcal{C}_{d,r}(\mathbb{F}_{q})} \leq c_{d,r}
    q^{-r+2}.
    $$
  \end{enumerate}
\end{theorem}
\begin{proof}
  Combining Lemma \ref{lemma:lower-bound-number-plane-curves} with
  Theorem \ref{th:upper-bound-red-curves}, the upper bounds follow
  immediately. It remains to prove the lower bounds.

  (1) Recall the morphism
  $\mu_{1,d,r}:\mathcal{C}_{1,r}\times\mathcal{C}_{d-1,r} \to
  \mathcal{C}_{d,r}$ induced by the multiplication mapping. Theorem
  \ref{th:dimension-variety-red-curves} asserts that
  $\mathrm{codim}_{\mathcal{R}_{d-1,r}} \mathcal{C}_{d-1,r}>0$, which
  implies that $\mathcal{C}_{d-1,r}\setminus\mathcal{R}_{d-1,r}$ is a
  nonempty Zariski open subset of $\mathcal{C}_{d-1,r}$ of dimension
  $b_{d-1,r}=\dim\mathcal{C}_{d-1,r.}$ Furthermore, the restriction of
  $\mu_{1,d,r}$ to $\mathcal{C}_{1,r}
  \times(\mathcal{C}_{d-1,r}\setminus \mathcal{R}_{d-1,r})$ is
  injective. Using Lemma \ref{lemma:lower-bound-number-plane-curves},
  it follows that
  \begin{align*}
    \#\mathcal{R}_{d,r}^{(q)}(\fq) & \ge \#\mathcal{C}_{1,r}(\fq)
    \cdot\#(\mathcal{C}_{d-1,r}(\fq)\setminus\mathcal{R}_{d-1,r}^{(q)}(\fq))\\
    & \ge \#\mathbb{G}(1,r)(\fq) \cdot\#({P}(d-1,r)(\fq)\setminus
    R_{d-1,r} (\fq))\\
    & > q^{2(r-1)} \cdot q^{3(r-2)+(d-1)(d+2)/2} \cdot (1- 13 q^{2-d}).
  \end{align*}
  The bound $\#\mathcal{C}_{d,r}(\fq)\le 2c_{d,r}q^{b_{d,r}}$ implies
  the inequality
  $$\frac{\#\mathcal{R}_{d,r}^{(q)}(\fq)}{\#\mathcal{C}_{d,r}(\fq)}
  \ge\frac{(1-13q^{2-d})}{2c_{d,r}}q^{-(d-2r+3)}.$$
  The last claim follows since $1-13 q^{2-d} \geq 1/2$ for $d \geq 7$.

  (2) We
  consider $dG(1,r)(\mathbb{F}_{q}) \subseteq \mathcal{R}_{d,r}^{(q)}
  (\mathbb{F}_{q})$ and the morphism from $G(1,r)^{d}$ to $dG(1,r)$
  which takes a sequence of lines to their sum. Each fiber of this
  morphism has size at most $d!$. This implies that
  \begin{align*}
    \#\mathcal{R}_{d,r}^{(q)}(\mathbb{F}_{q}) \geq \#
    dG(1,r)(\mathbb{F}_{q}) > \frac{(\#G(1,r)(\mathbb{F}_{q}))^{d}}{d!}
    \geq \frac{q^{2d(r-1)}}{d!}.
  \end{align*}

  Combined with Lemma \ref{lemma:lower-bound-number-plane-curves}, this
  yields
  $$
  \frac{\#\mathcal{R}_{d,r}^{(q)}(\mathbb{F}_{q})}{\#\mathcal{C}_{d,r}(\mathbb{F}_{q})}
  \geq \frac{q^{2d(r-1)-b_{d,r}}}{2 ~ d! c_{d,r}}.
  $$
  Furthermore,
  $$
  2d(r-1)-b_{d,r}= 8(r-2)(r-1)- (3(r-2)+2(r-2)(4r-5))=-r+2.
  $$
This finishes the proof of the theorem.\qed\end{proof}

  An immediate consequence of this theorem is that the probability that an
  $\fq$-curve in $\Pp^r$ of degree $d$ is $\fq$-reducible tends to zero
  as $q$ grows for fixed $r\ge 3$ and $d\ge 4r-8$.  Furthermore, our
  bounds show that such a convergence has the same rate as
  $q^{-(d-2r+3)}$. In this sense, our bounds are a suitable
  generalization of the corresponding bounds for $r=2$, as stated in
  (\ref{eq:lower_bound_irred_proj_plane_curves}). We made no attempt to
  optimize the ``constants'' independent of $q$.

  In \cite{EiHa92} it is proved that for $d \geq 4r-8$, the planar curves in
  $P(d,r)$ form the only component of $\mathcal{C}_{d,r}$ with maximal
  dimension. In this sense, ``most'' curves are planar. We can quantify
  this as follows.
  \begin{corollary}\label{cor:EiHa}
    For $r \geq 3$ and $d \geq 4r - 8$,
    \begin{equation*}
      \frac{\#(\mathcal{C}_{d,r} \backslash
        P(d,r))(\mathbb{F}_{q})}{\#\mathcal{C}_{d,r}(\mathbb{F}_{q})} \leq \frac{2c_{d,r}}{q}.
    \end{equation*}
  \end{corollary}

  \begin{proof}
    We denote as $N$ the union of all components of $\mathcal{C}_{d,r}$
    not contained in $P(d,r)$. Thus all non-planar curves are in $N$. By
    Fact \ref{fact:dimension-chow-var}, it follows that $\dim N < b_{d,r}$, and
    since each component of $N$ is a component of $\mathcal{C}_{d,r}$,
    we have $\deg N \leq \deg \mathcal{C}_{d,r} = c_{d,r}$. Now
    (\ref{eq:upper-bound-number-rat-points}) implies that
    $\#N(\mathbb{F}_{q}) \leq 2c_{d,r} q^{b_{d,r}-1}$, and using
    Lemma \ref{lemma:lower-bound-number-plane-curves}
    $$
    \frac{\#((\mathcal{C}_{d,r} \setminus
      P(d,r))(\mathbb{F}_{q}))}{\#\mathcal{C}_{d,r}} \leq
    \frac{\#(N(\mathbb{F}_{q}))}{q^{b_{d,r}}} \leq \frac{2c_{d,r}}{q}.
    $$
This shows the corollary.\qed\end{proof}
  In particular, for fixed $r\ge 3$ and $d \geq 4r-8$, the probability for
  a random curve to be non-planar tends to $0$ with growing $q$.


  %
  %
  \section{The probability that an $\fq$-curve is absolutely
    reducible}\label{sec:AbsRed}
  An $\fq$-curve can be absolutely reducible for two reasons: either it is
  $\fq$-reducible, as treated above, or \emph{relatively $\fq$-irreducible},
  that is, is $\fq$-irreducible and $\cfq$-reducible. The aim of this section
  is to obtain a bound on the probability for the latter to occur.
  The set of relatively $\fq$-irreducible (or
  exceptional) $\fq$-curves of degree $d$ in $\Pp^r$ is denoted by
  $\mathcal{E}_{d,r}(\fq)$ and the set of irreducible $\fq$-curves of
  degree $d$ in $\Pp^r$ is denoted by $\mathcal{I}_{d,r}(\fq)$.

  \begin{theorem}\label{th:upper-bound-rel-red-curves}
    Let $r\ge 3$ and $d\ge 4r-8 $, and denote by $\ell$ the smallest
    prime divisor of $d$. We have the bounds
    $$  \begin{array}{rcccccc}
      q^{2d(r-1)}(1-4q^{2(1-d)(r-1)}) \!\!\!\! &\le \#\mathcal{E}_{d,r}(\fq) \le 
      2D_{\ell,d,r}\,q^{2d(r-1)}
      &\text{ for } d/\ell\le 4r-7,  \\[1ex]
      q^{\ell b_{d/\ell,r}}(1- 16q^{\ell-d}) \!\!\!\! &\le
      \#\mathcal{E}_{d,r}(\fq) \le  3D_{\ell,d,r}\,q^{\ell b_{d/\ell,r}} &
      \text{ for } d/\ell\ge 4r-8,
    \end{array}$$
    with $D_{\ell,d,r}\defineAs
    (ed/\ell)^{r(r+1)(d^2/\ell^2+1)+4rg_{d/\ell,r}}$, $b_{d,r}\defineAs
    \dim\mathcal{C}_{d,r}$ and $g_{d/\ell,r}$ as in
    (\ref{eq:dim-homogeneous-piece-grassmannian}).
  \end{theorem}

  \begin{proof}
    We follow the lines of the proof of \cite[Theorem 5.1]{Gathen08}.
    Let $A_0\klk A_r,B_0\klk B_r$ be new indeterminates, let $\boldsymbol{A}\defineAs
    (A_0\klk A_r)$ and $\boldsymbol{B}\defineAs (B_0\klk B_r)$. First we observe
    that, if a bihomogeneous polynomial $F\in\fq[\boldsymbol{A},\boldsymbol{B}]$ of bidegree
    $(d,d)$ is relatively $\fq$-irreducible, then it is reducible in
    $\fqk$ for $k$ dividing $d$. Therefore, let $k$ be a divisor of $d$
    and let $\mathcal{G}_k\defineAs \hbox{Gal}(\fqk:\fq)$ be the Galois
    group of $\fqk$ over $\fq$. For $\sigma$ in $\mathcal{G}_k$ and
    $[F]$ in $\mathcal{C}_{d/k,r}(\fqk)$, the application of $\sigma$ to
    the coordinates of $[F]$ yields a point $[F^\sigma]$ in
    $\Pp\V_{d/k,r}$. Moreover, we have the following claim:
    \begin{claim} $[F^\sigma]$ belongs to $\mathcal{C}_{d/k,r}(\fqk)$,
      i.e., there is an $\fqk$-cycle of dimension 1 and degree $d/k$ of
      $\Pp^r$ that corresponds to $[F^\sigma]$.\end{claim}
    \begin{poc} Any morphism in $\mathcal{G}_k$ can be extended (not
      uniquely) to a morphism $\widetilde{\sigma}$ in
      $\hbox{Gal}(\cfq:\fq)$. By Theorem
      \ref{th:correspondence-chow-points-and-curves} we see that the
      support $\mathrm{supp}(F)\subset\Pp^r$ of $[F]$ is an
      $\fqk$-curve. Applying $\widetilde{\sigma}$ to the coordinates of
      the points of $\mathrm{supp}(F)$, by the $\fqk$-definability of
      $\mathrm{supp}(F)$ we deduce the equality
      $\mathrm{supp}(F^\sigma)=\sigma(\mathrm{supp}(F))$.  This shows
      that $\mathrm{supp}(F^\sigma)$ is an $\fqk$-curve in $\Pp^r$,
      which in turn proves that $[F^\sigma]\in\mathcal{C}_{d/k,r}(\fqk)$.
    \end{poc}

    The previous claim shows that the following mapping is well-defined:
    $$
    \begin{array}{crcl}
      \varphi_{k,d}\colon&\mathcal{C}_{d/k,r}(\fqk)&\to&\mathcal{C}_{d ,r}(\fq ) \\
      &[F]&\mapsto&\left[\prod_{\sigma\in \mathcal{G}_k}F^\sigma\right].
    \end{array}
    $$
    %
    The image $\varphi_{k,d}([F])$ of the class of an $\fqk$-irreducible
    polynomial $F$ is $\fq$-reducible if and only if there exists a proper
    divisor $l$ of $k$ such that $[F]$ is
    $\fqll$-definable. Furthermore, if $[F]$ is relatively
    $\fqll$-irreducible, then $\varphi_{k,d}([F])=\varphi_{j,d}([H])$ for
    an appropriate multiple $j$ of $k$ and $[H]$ in
    $\mathcal{I}_{d,r}(\F_{\hskip-0.7mm q^j})$. Thus, if we set for any
    integer $m$
    \begin{align*}
      \mathcal{I}_{m,r}^+(\fqk:\fq) &\defineAs
      \mathcal{I}_{m,r}(\fqk)\setminus
      \big(\mathcal{E}_{m,r}(\fqk)\cup\bigcup_{s>1,s|k}
      \mathcal{I}_{m,r}(\F_{\hskip-0.7mm q^{k/s}})\big),
      \\
      \mathcal{E}_{k,d,r}&\defineAs
      \varphi_{k,d}(\mathcal{I}_{d/k,r}^+(\fqk:\fq)),
    \end{align*}
    then we have the equality
    $$\mathcal{E}_{d,r}(\fq)=\bigcup_{k>1,k|d}\mathcal{E}_{k,d,r}.$$

    In order to obtain bounds on the cardinality of
    $\mathcal{E}_{d,r}(\fq)$, we observe that, for any divisor $k$ of
    $d$ with $k>1$, we have
    \begin{displaymath}
      \#\mathcal{E}_{k,d,r}= \#\mathcal{I}_{d/k,r}^+(\fqk:\fq).
    \end{displaymath}
    Therefore,
    \begin{equation}\label{eq:aux-upper-bound-for-rel-reds}
      \#\mathcal{I}_{d/l,r}^+(\fqll)\le \#\mathcal{E}_{d,r}(\fq)
      \le\sum_{k>1,k|d}\#\mathcal{I}_{d/k,r}^+(\fqk) \le\sum_{k>1,
        k|d}\#\widehat{\mathcal{C}}_{d/k,r}(\fqk)
    \end{equation}
    for any divisor $l>1$ of $d$.

    %
    %

    The case $d$ prime follows directly from this expression, since the
    sum in the right--hand side consists of only one term, namely
    $\#\mathcal{E}_{d,r}(\fq)=\#\mathcal{I}_{1,r}^+(\fqd)$.
    Furthermore,
    %
    %
    %
    \begin{eqnarray*}\#\mathcal{I}_{1,r}^+(\fqd:\fq)=
      \#\left(\mathcal{I}_{1,r}(\fqd)\setminus
        \mathcal{I}_{1,r}(\fq)\right)&=&\#\left(\mathbb{G}_{1,r}(\fqd)\setminus
        \mathbb{G}_{1,r}(\fq)\right)\\
      &=&q^{2d(r-1)}-\frac{(q^{r+1}-1)(q^{r+1}-q)}{(q^2-1)(q^2-q)}.
    \end{eqnarray*}
    Hence, for $d$ prime we have
    \begin{equation}
      \label{eq:bound-prob-abs-red-curves-case-prime}
      q^{2d(r-1)}(1-4q^{2(1-d)(r-1)})\le\#\mathcal{E}_{d,r}(\fq)\le
      q^{2d(r-1)}.
    \end{equation}

    Now, assume that $d$ is not prime, and let $\ell$ denote the
    smallest prime divisor of $d$. Suppose that $d/\ell\le 4r-7$. In
    this case, from Fact \ref{fact:dimension-chow-var} we have
    $b_{d/k,r}=2d(r-1)/k$ for every divisor $k$ of $d$. As a
    consequence, if we denote
    $$D_{k,d,r}\defineAs (ed/k)^{r(r+1)(d^2/k^2+1)+4rg_{d/k,r}},$$
    from (\ref{eq:aux-upper-bound-for-rel-reds}) and Proposition
    \ref{prop:degree-chow-var-irred-curves} we see that

    $$
    \#\mathcal{E}_{d,r}(\fq)\le
    \displaystyle \sum_{k>1,k|d}2D_{k,d,r}\,q^{kb_{d/k,r}} \le
    D_{\ell,d,r}\,q^{2d(r-1)}
    \sum_{k>1,k|d}2\frac{D_{k,d,r}}{D_{\ell,d,r}}.
  $$
  For a nontrivial divisor $k>\ell$ of $d$, we have
  \begin{gather}
    \begin{aligned}
      \frac{D_{k,d,r}}{D_{\ell,d,r}}
      \le
      \left(\frac{ed}{\ell}\right)^{(d^{2}/k^{2}-d^{2}/\ell^{2})r(r+1)}
      \le \left(\frac{ed}{\ell}\right)^{-2r(r+1)}\le \frac{1}{2d}.
    \end{aligned}
    \label{eq:auxiliar-1-prob-rel-red-curves}
  \end{gather}
  We conclude that, for $d/\ell\le 4r-7$, the following upper bound
  holds:
  $$
  \#\mathcal{E}_{d,r}(\fq)\le 2D_{\ell,d,r}\,q^{2d(r-1)}.$$
  In order to determine a ``matching'' lower bound, arguing as above
  we obtain
  $$\#\mathcal{E}_{d,r}(\fq)\ge \#\mathcal{E}_{d,d,r}=\#\mathcal{I}_{1,r}^+(\fqd:\fq)
  \ge q^{2d(r-1)}-4q^{2(r-1)}.$$
  Summarizing, we have
  \begin{equation}\label{eq:bound-prob-abs-red-curves-case-2}
    q^{2d(r-1)}(1-4q^{2(1-d)(r-1)})\le
    \#\mathcal{E}_{d,r}(\fq)\le 2D_{\ell,d,r}\,q^{2d(r-1)}.
  \end{equation}

  Finally, assume that $d/\ell\ge 4r-8$. Then Fact
  \ref{fact:dimension-chow-var} implies $b_{d/k,r}=3(r-2)+d(d/k+3)/2k$
  for $d/k\ge 4r-8$ and $b_{d/k,r}=2d(r-1)/k$ for $d/k\le 4r-7$.

  \begin{claim}
    For any divisor $k>\ell$ of $d$, we have
    \begin{equation}\label{eq:aux-bound-rel-red-curves}
      kb_{d/k,r}<\ell b_{d/\ell,r}\defineAs 3\ell(r-2)+d(d/\ell+3)/2.
    \end{equation}
  \end{claim}
  \begin{poc}

    First we consider the case ${d}/({4r-8})\ge k$. Then we
    have $kb_{d/k,r}=3k(r-2)+d(d/k+3)/2$. Taking formal derivatives in
    this expression with respect to $k$, we conclude that $k\mapsto
    kb_{d/k,r}$ is a strictly decreasing function of $k$ for
    ${d}/({4r-8})\ge k$. This implies the claim in this case.

    Next assume that $d/k\le 4r-7$. Then we have $kb_{d/k,r}
    =2d(r-1)$. Up to a division by $\ell$, we see that the claim is
    equivalent to the validity of the inequality
    $$2(d/\ell)(r-1)<3(r-2)+(d/\ell)(d/\ell+3)/2.$$
    Then Fact \ref{fact:dimension-chow-var} shows that the last
    inequality holds for $d/\ell\ge 4r-8>1$. This concludes the proof of
    our claim.
  \end{poc}
  From (\ref{eq:aux-upper-bound-for-rel-reds}) and
  (\ref{eq:aux-bound-rel-red-curves}) it follows that
  $$
  \#\mathcal{E}_{d,r}(\fq)\le \sum_{k>1,k|d}2D_{k,d,r}\,q^{kb_{d/k,r}}
  \le 2 D_{\ell,d,r}\,q^{\ell b_{d/\ell,r}}
  \bigg(1+q^{-1}\sum_{k>\ell,k|d} \frac{D_{k,d,r}}{D_{\ell,d,r}}
  \bigg).
  $$
  Applying (\ref{eq:auxiliar-1-prob-rel-red-curves}) 
  %
  %
  we obtain, for $d/\ell\ge 4r-8$,
  \begin{equation}
    \#\mathcal{E}_{d,r}(\fq)
    \le 2D_{\ell,d,r}\,q^{\ell b_{d/\ell,r}} (1+(2q)^{-1}) \le
    3D_{\ell,d,r}\,q^{\ell b_{d/\ell,r}}.
    \label{eq:upper-bound-prob-abs-red-curves-case-3}
  \end{equation}

  Next we obtain a lower bound for this case. We have
  $$\#\mathcal{E}_{d,r}(\fq)\ge \#\mathcal{E}_{\ell,d,r}=
  \#\mathcal{I}_{d/\ell,r}^+(\fql:\fq)=\#\left(\mathcal{I}_{d/\ell,r}(\fql)
    \setminus\mathcal{I}_{d/\ell,r}(\fq)\right),$$
  since $\ell$ is prime and there are no proper intermediate fields
  between $\fq$ and $\fql$. In order to find a lower bound for
  right--hand side above, we observe that
  $$\#\left(\mathcal{I}_{d/\ell,r}(\fql)
    \setminus\mathcal{I}_{d/\ell,r}(\fq)\right)\ge
  \#\left(\left(\mathcal{I}_{d/\ell,r}(\fql)
      \setminus\mathcal{I}_{d/\ell,r}(\fq)\right)\cap
    P(d/\ell,r)(\fql)\right).$$
  According to Lemma \ref{lemma:lower-bound-number-plane-curves}, we
  have
  \begin{eqnarray*}\#\left(\mathcal{I}_{d/\ell,r}(\fql) \cap
      P(d/\ell,r)(\fql)\right)&=&\# P(d/\ell,r)(\fql)-
    \#R(d/\ell,r)(\fql)\\ &\ge&q^{\ell b_{d/\ell,r}}(1- 15q^{\ell-d}).
  \end{eqnarray*}
  On the other hand, Lemma \ref{lemma:lower-bound-number-plane-curves}
  implies
  $$\#\left(\mathcal{I}_{d/\ell,r}(\fq)\cap
    P(d/\ell,r)(\fql)\right)\le \#P(d/\ell,r)(\fq)\le
  8q^{b_{d/\ell,r}}.$$
  As a consequence, it follows that
  \begin{equation}\#\mathcal{E}_{d,r}(\fq)\ge q^{\ell b_{d/\ell,r}}(1-
    15q^{\ell-d}-8q^{(1-\ell)b_{d/\ell,r}}) \ge q^{\ell
      b_{d/\ell,r}}(1-
    16q^{\ell-d}). \label{eq:lower-bound-prob-abs-red-curves-case-3}\end{equation}
  Combining (\ref{eq:upper-bound-prob-abs-red-curves-case-3}) and
  (\ref{eq:lower-bound-prob-abs-red-curves-case-3}), we obtain
  \begin{equation}\label{eq:bound-prob-abs-red-curves-case-3}
    q^{\ell b_{d/\ell,r}}(1-
    16q^{\ell-d})\le \#\mathcal{E}_{d,r}(\fq)\le 3D_{\ell,d,r}\,q^{\ell b_{d/\ell,r}}.
  \end{equation}

  Putting together (\ref{eq:bound-prob-abs-red-curves-case-prime}),
  (\ref{eq:bound-prob-abs-red-curves-case-2}), and
  (\ref{eq:bound-prob-abs-red-curves-case-3}) finishes the proof of
  the theorem.
\qed\end{proof}
Arguing as in the proof of Corollary
\ref{coro:upper-bound-prob-red-curves}, we obtain the following
consequence of Theorem \ref{th:upper-bound-rel-red-curves}, again
with an exact rate of convergence in $q$.
\begin{corollary}
  \label{coro:upper-bound-prob-rel-red-curves}
  With notations and assumptions as in Theorem
  \ref{th:upper-bound-rel-red-curves}, we have
  $$
  \begin{array}{r}
    \displaystyle\frac{(1-4q^{2(1-d)(r-1)})}{2c_{d,r}}\,q^{(2d-3)(r-2)-\frac{d(d-1)}{2}}\le
    \frac{ \#\mathcal{E}_{d,r}(\fq)}{ \#\mathcal{C}_{d,r}(\fq)}\le
    2D_{\ell,d,r}
    \,q^{(2d-3)(r-2)-\frac{d(d-1)}{2}}\\\textit{for}\ d/\ell\le
    4r-7,\\[1ex]
    \displaystyle\frac{(1-16q^{\ell-d})}{2c_{d,r}}\,
    q^{3(\ell-1)(r-2)-d^2(\ell-1)/2\ell}
    \le\frac{ \#\mathcal{E}_{d,r}(\fq)}{ \#\mathcal{C}_{d,r}(\fq)}\le
    3D_{\ell,d,r}
    \,q^{3(\ell-1)(r-2)-d^2(\ell-1)/2\ell}\\\textit{for}\ d/\ell\ge 4r-8,\end{array}$$
  with $D_{\ell,d,r}\defineAs
  (ed/\ell)^{r(r+1)(d^2/\ell^2+1)+4rg_{d/\ell,r}}$,
  $c_{d,r}=(2ed)^{r(r+1)(d^2+1)+4rg_{d,r}}$ and $g_{d/\ell,r}$ as in
  (\ref{eq:dim-homogeneous-piece-grassmannian}).
\end{corollary}
%
%
%

\section{The average number of $\fq$-rational points on
  $\fq$-curves}\label{sec:Weil}
The present paper was partially motivated by the following question:
how many rational points does a typical curve have? As a consequence
of the seminal paper of A. Weil \cite{Weil48}, for an absolutely
irreducible $\fq$-curve $C$ of degree $d$ of $\Pp^r$, we have the
estimate (see, e.g., \cite{Schmidt76})
\begin{equation}\label{eq:Weil estimate}
  |\#C(\fq)-(q+1)|\le (d-1)(d-2)q^{{1/2}}+\lambda(d,r),
\end{equation}
where $\lambda(d,r)$ is a constant independent of $q$. From
\cite{CaMa06} it follows that we can take $\lambda(d,r)=6d^2$ if
$q\ge 15d^{13/3}$. Combining these inequalities yields
\begin{equation}\label{eq:Weil-Cafure-Matera estimate}
  |\#C(\fq)-(q+1)|\le d^2q^{{1/2}}
\end{equation}
for any absolutely irreducible $\fq$-curve and $q\ge 15d^{13/3}$.
Recall that Corollaries \ref{coro:upper-bound-prob-red-curves} and
\ref{coro:upper-bound-prob-rel-red-curves} assert that ``almost
all'' curves are absolutely irreducible for large values of $q$. The
set of $\fq$-curves $C$ of $\Pp^r$ of degree $d$ satisfying
(\ref{eq:Weil-Cafure-Matera estimate}) contains the set of
absolutely irreducible curves of $\mathcal{C}_{d,r}(\fq)$. From
these remarks we obtain the following result on the average number
of $\fq$-rational points of the curves in $\mathcal{C}_{d,r}(\fq)$.
\begin{theorem}
  Let notation be as in Theorem \ref{th:upper-bound-rel-red-curves}
  and assume that $q\ge 15d^{13/3}$, $r\ge 3$ and $d>4r-7$. Then the
  expectation of $\#C(\fq)$ for uniformly random $C$ in $\mathcal{C}_{d,r}(\fq)$
  satisfies
  \begin{equation}\label{eq:expectation-rat-points}
    \big|{\mathbb E} [\#C(\fq)]-(q+1)\big|\le d^2q^{1/2}+
    3\,d\, c_{d,r}\,q^{-(d-2r+2)}
  \end{equation}
  with $c_{d,r}\defineAs \left(2ed \right)^{r(r+1)(d^2+1)+4rg_{d,r}}$.
  Moreover, the probability distribution is concentrated around the
  expectation, namely
  \begin{equation}\label{eq:proba-concentration}
    \Pr\big[|\#C(\fq)-(q+1)|\le d^2q^{1/2}\big]\ge
    1-2\,c_{d,r}\,q^{-(d-2r+3)}.
  \end{equation}
  The latter bound tends to $1$ as $q$ tends to infinity.
\end{theorem}
\begin{proof} First we prove (\ref{eq:proba-concentration}). Let
  $\mathcal{A}_{d,r}(\fq)$ denote the set of absolutely irreducible
  $\fq$-curves. This set is the complement in $\mathcal{C}_{d,r}(\fq)$
  of the union of the set $\mathcal{R}_{d,r}^{(q)}(\fq)$
  of $\fq$-reducible $\fq$-curves plus the set $\mathcal{E}_{d,r}(\fq)$
  of relatively irreducible $\fq$-curves. Hence we have
  $$
  \Pr[\mathcal {A}_{d,r}(\fq)]\ge
  1-2\max\{\Pr[\mathcal{R}_{d,r}^{(q)}(\fq)],\Pr[\mathcal{E}_{d,r}(\fq)]\}.
  $$
  The assumption on $d$ implies
  \begin{equation}
    \label{eq:aux-theorem-average-number-of-points}
    \min\left\{d(d-1)/2-(2d-3)(r-2),d^2(\ell-1)/2\ell-3(\ell-1)(r-2)\right\}\ge
    d-2r+3.
  \end{equation}
  From Theorem \ref{coro:upper-bound-prob-red-curves},
  Corollary \ref{coro:upper-bound-prob-rel-red-curves} and
  (\ref{eq:aux-theorem-average-number-of-points}), it follows that
  \begin{equation}\label{eq:aux_proba_reducible-relirred}
    \max\{\Pr[\mathcal{R}_{d,r}^{(q)}(\fq)],\Pr[\mathcal{E}_{d,r}(\fq)]\}\le
    c_{d,r}\,q^{-(d-2r+3)}.
  \end{equation}
  Finally, (\ref{eq:Weil-Cafure-Matera estimate}) and (\ref{eq:aux_proba_reducible-relirred}) yield
  %
  $$
  \Pr\big[|\#C(\fq)-(q+1)|\le d^2q^{1/2}\big]\ge
  \Pr[\mathcal{A}_{d,r}(\fq)]\ge 1-2\,c_{d,r}\,q^{-(d-2r+3)}.
  $$

  Now we estimate the expectation
  (\ref{eq:expectation-rat-points}). For this purpose we observe that
  (\ref{eq:upper-bound-number-rat-points}) implies $\#C(\fq)\le
  d(q+1)$ for any curve $C\in\mathcal{C}_{d,r}(\fq)$. Combining this
  upper bound with (\ref{eq:aux_proba_reducible-relirred}) we obtain
  \begin{gather}
    \begin{aligned}
      {\mathbb E} [\#C(\fq)]& \le (q+1+d^2q^{1/2})\Pr[\#C(\fq)\le q+1
      +d^2q^{1/2}]\\& \quad\ + d(q+1)\Pr[\#C(\fq)>q+1
      +d^2q^{1/2}]\nonumber
      \\ & \le q+1+d^2q^{1/2}+  d(q+1)\Pr[\mathcal{E}_{d,r}(\fq)\cup
      \mathcal{R}_{d,r}^{(q)}(\fq)]\nonumber \\
      & \le q+ 1 +d^2q^{1/2}+3\,d\,c_{d,r}\,q^{-(d-2r+2)}.
    \end{aligned}
  \end{gather}%
  On the other hand, we have
  \begin{gather}
    \begin{aligned}
      {\mathbb E} [\#C(\fq)]& \ge (q+1-d^2q^{1/2})\Pr[\#C(\fq)\ge q+1
      -d^2q^{1/2}]\nonumber
      \\ & \ge (q+1-d^2q^{1/2})\Pr[ \mathcal{A}_{d,r}(\fq)]\nonumber \\
      & \ge q+ 1 -d^2q^{1/2}-2\,c_{d,r}\,q^{-(d-2r+2)}.
    \end{aligned}
  \end{gather}%
  Combining the upper and the lower bound on ${\mathbb E}
  [\#C(\fq)]$, we deduce (\ref{eq:expectation-rat-points}).
\qed\end{proof}

\noindent{\bf Open question.} Can one similarly determine the
probabilities for other ``rare'' types of curves, say, the ones that
are singular or not complete intersections?

\end{document}






%
\end{document}
